\documentclass{amsart}

\usepackage{microtype}
\usepackage{amscd}
\usepackage{amsmath,amssymb,amsfonts, accents}
\usepackage{xstring}
\usepackage{xcolor}
\usepackage[mathscr]{euscript}
\usepackage{xypic}
\usepackage[cmtip,all,2cell]{xy}
\entrymodifiers={+!!<0pt,\fontdimen22\textfont2>}
\usepackage{url}
\usepackage{latexsym}
\usepackage{amsthm}
\usepackage[latin1]{inputenc}
\usepackage{multicol}
\usepackage{enumitem}
\usepackage[colorlinks=true, linkcolor={blue!50!black}, pdfhighlight=/O, ocgcolorlinks=true]{hyperref}
\usepackage{graphicx}
\usepackage{color}

\numberwithin{equation}{section}
\theoremstyle{plain}
\newtheorem{theorem}[equation]{Theorem}

\newtheorem{lemma}[equation]{Lemma}

\newtheorem{proposition}[equation]{Proposition}

\newtheorem{corollary}[equation]{Corollary}

\theoremstyle{definition}

\newtheorem{definition}[equation]{Definition}

\newtheorem{construction}[equation]{Construction}
\newtheorem{example}[equation]{Example}

\newtheorem{remark}[equation]{Remark}
\newtheorem{warning}[equation]{Warning}

\newtheoremstyle{TheoremNum} 
        {\topsep}{\topsep}              
        {\itshape}                      
        {}                              
        {\bfseries}                     
        {.}                             
        { }                             
        {\thmname{#1}\thmnote{ \bfseries #3}}
    \theoremstyle{TheoremNum}

\newenvironment{smitemize} 
{ \begin{itemize}
    \setlength{\itemsep}{0pt}
    \setlength{\parskip}{0pt}
    \setlength{\parsep}{0pt}     }
{ \end{itemize}                  } 
{ \begin{enumerate}
    \setlength{\itemsep}{0pt}
    \setlength{\parskip}{0pt}
    \setlength{\parsep}{0pt}     }
{ \end{enumerate}                  }

\newcommand{\field}[1]{\mathbb{#1}}
\newcommand{\mf}[1]{\mathfrak{#1}}
\newcommand{\mc}[1]{\mathcal{#1}}

\newcommand{\cat}[1]{
\StrLen{#1}[\mystrlen]
\ifnum\mystrlen=1 \mathscr{#1}
\else \mathrm{#1}
\fi}

\newcommand{\oocat}[1]{
\StrLen{#1}[\mystrleng]
\ifnum\mystrleng=1 \mathscr{#1}
\else \mathrm{#1}
\fi
}
\newcommand{\ope}[1]{
\StrLen{#1}[\mystrlengt]
\ifnum\mystrlengt=1 \mathscr{#1}
\else \mathrm{#1}
\fi
}

\newcommand{\Set}[0]{\oocat{Set}}

\newcommand{\sS}[0]{\cat{S}}
\newcommand{\colim}{\operatornamewithlimits{\mathrm{colim}}}

\newcommand{\dau}[0]{\partial}

\newcommand{\mm}[1]{\mathrm{#1}}

\newcommand{\Hom}[0]{\mm{Hom}}
\newcommand{\Map}[0]{\mm{Map}}
\renewcommand{\rto}[1]{\stackrel{#1}{\rt}}

\newcommand{\rt}[0]{\rightarrow}

\newcommand{\spec}[0]{\mm{Spec}}

\newcommand{\Fun}[0]{\cat{Fun}}

\newcommand{\op}[0]{\mm{op}}
\newcommand{\dR}[0]{\mm{dR}}
\newcommand{\QC}[0]{\mm{QC}}
\newcommand{\Mod}[0]{\mm{Mod}}
\newcommand{\dgMod}[0]{\Mod^{\dg}}

\newcommand{\Rep}[0]{\cat{Rep}}
\newcommand{\dgRep}[0]{\Rep^{\dg}}
\newcommand{\sgeq}[0]{{\scriptscriptstyle \geq}}
\newcommand{\sleq}[0]{{\scriptscriptstyle \leq}}

\newcommand{\dgCAlg}[0]{\CAlg^\mm{\dg}}
\newcommand{\dg}[0]{\mm{dg}}

\newcommand{\CAlg}[0]{\cat{CAlg}}
\newcommand{\LieAlgd}[0]{\cat{LieAlgd}}
\newcommand{\dgLieAlgd}[0]{\LieAlgd^{\dg}}

\DeclareMathSymbol{\widetildesym}{\mathord}{largesymbols}{"65}

\DeclareMathSymbol{\widehatsym}{\mathord}{largesymbols}{"62}

\title{Koszul duality for Lie algebroids}
\author{Joost Nuiten}
\email{j.j.nuiten@uu.nl}
\address{Mathematical Institute\\ Utrecht University\\ P.O. Box 80010\\ 3508 TA Utrecht\\ The
Netherlands.}

\begin{document}
\begin{abstract}
This paper studies the role of dg-Lie algebroids in derived deformation theory. More precisely, we provide an equivalence between the homotopy theories of formal moduli problems and dg-Lie algebroids over a commutative dg-algebra of characteristic zero. At the level of linear objects, we show that the category of representations of a dg-Lie algebroid is an extension of the category of quasi-coherent sheaves on the corresponding formal moduli problem. We describe this extension geometrically in terms of pro-coherent sheaves.
\end{abstract}

\maketitle


\section{Introduction}\label{sec:intro}
Lie algebroids appear throughout algebraic and differential geometry as natural objects describing infinitesimal geometric structures, like foliations and actions of Lie algebras \cite{rin63, mac87}. In this paper, we study the relation between Lie algebroids and deformation theory, which can informally be described as follows: suppose that $k$ has characteristic zero and that
$$\xymatrix{
x\colon \spec(A)\ar[r] & X
}$$
is a map from an affine (derived) scheme to a moduli space over $k$. Then a formal neighbourhood of $X$ around $x$ is controlled by a Lie algebroid over $\spec(A)$. This Lie algebroid is given by the vector fields on $\spec(A)$ that are (derived) tangent to the fibers of $x$. 

When $\spec(A)$ is a point, this recovers the well-known relation between deformation problems and dg-Lie algebras, which has been extensively studied and applied \cite{gol88, kon03, man02}, and has been given a precise mathematical formulation in the work of Hinich \cite{hin98}, Pridham \cite{pri10} and Lurie \cite{lur11X}. An important idea in these works is to describe the formal neighbourhood of a moduli space in terms of derived geometry, using its functor of points. More precisely, one can describe a formal neighbourhood of $X$ around $x$ by a functor
$$\xymatrix{
X^\wedge\colon \CAlg_k^\mm{sm}/A\ar[r] & \sS; \hspace{4pt} B\ar@{|->}[r] & X(B)\times_{X(A)} \{x\}
}$$
sending a \emph{derived infinitesimal thickening} $\spec(A)\rt \spec(B)$ to the space of extensions of $x$ to a map $\tilde{x}\colon \spec(B)\rt X$. For any reasonable moduli space $X$, the functor $X^\wedge$ satisfies a derived version of the Schlessinger conditions, which encodes the usual obstruction theory for existence and uniqueness of deformations \cite{sch68}. The works mentioned above provide an equivalence between the homotopy theory of such functors satisfying the Schlessinger conditions and the homotopy theory of dg-Lie algebras, at least when $A$ is a field (see also \cite{hen15} for an extension to dg-Lie algebras over more general dg-algebras $A$).

The purpose of this paper is to provide a similar identification of the homotopy theory of dg-Lie algebroids with the homotopy theory of \emph{formal moduli problems}, in the sense of \cite{lur11X, gai16}:
\begin{definition}\label{def:formalmoduli}
Let $A$ be a connective commutative $k$-algebra, where $k$ has characteristic zero. The $\infty$-category $\cat{CAlg}_k^\mm{sm}/A$ of \emph{small extensions} of $A$ is the smallest subcategory of commutative $k$-algebras over $A$ that contains $A$ and that is closed under square zero extensions by $A[n]$ with $n\geq 0$. 

A \emph{formal moduli problem} over $A$ is a functor
$$\xymatrix{
X\colon \CAlg_k^\mm{sm}/A\ar[r] & \sS
}$$
satisfying the following two conditions:
\begin{itemize}
\item[(a)] $X(A)\simeq \ast$ is contractible.
\item[(b)] $X$ preserves any pullback diagram of small extensions of $A$ of the form
\begin{equation}\label{diag:squarezerointro}\vcenter{\xymatrix{
B_\eta\ar[d]\ar[r] & A\ar[d]^{(\mm{id}, 0)}\\
B\ar[r]_-\eta & A\oplus A[n+1].
}}\end{equation}
Such a pullback square realizes $B_\eta$ as a square zero extension of $B$ by $A[n]$ (for $n\geq 0$).
\end{itemize}
\end{definition}
Geometrically, a formal moduli problem can be thought of as a map of stacks $x\colon \spec(A)\rt X$ that realizes $X$ as an infinitesimal thickening of $\spec(A)$. We show that such a stack determines a Lie algebroid $T_{A/X}$, which can be thought of as the fiberwise vector fields on $\spec(A)$ over $X$. This construction is part of an equivalence, with inverse sending a Lie algebroid $\mf{g}$ to the `quotient' of $\spec(A)$ by the infinitesimal $\mf{g}$-action:
\begin{theorem}\label{thm:intro1}
Suppose that $A$ is eventually coconnective. Then there is an equivalence of $\infty$-categories
$$\xymatrix{
\mm{MC}\colon \LieAlgd_A\ar@<1ex>[r] & \cat{FormMod}_A\ar@<1ex>[l] \colon T_{A/}
}$$
between the $\infty$-category of Lie algebroids over $A$ and the $\infty$-category of formal moduli problems over $A$.
\end{theorem}
The $\infty$-category of Lie algebroids has an explicit description in terms of homological algebra, as the localization of the category of dg-Lie algebroids over $A$ at the quasi-isomorphisms. One can therefore think of the above result as a rectification result, showing that any formal moduli problem admits a rigid description in terms of chain complexes endowed with algebraic structure. For example, the formal moduli problem describing the deformations of a connective commutative $A$-algebra $R$ admits a simple chain-level description, as the dg-Lie algebroid of compatible derivations on both $A$ and $R$ (Proposition \ref{prop:defofalgebras}, cf.\ \cite{hin04, toe16}). 

The proof of Theorem \ref{thm:intro1} follows the lines of \cite{lur11X}, and relies on a version of Koszul duality: the small extensions of $A$ are Koszul dual to certain free Lie algebroids over $A$, by means of the functor sending a dg-Lie algebroid to its cohomology (as studied already in \cite{rin63}). To make efficient use of this result, it is useful to study Koszul duality at the level of \emph{linear} objects as well. More precisely, given a quasi-coherent sheaf $\mc{F}$ on $X$, one expects the restriction $x^*\mc{F}$ along $x\colon \spec(A)\rt X$ to carry a natural representation of the Lie algebroid $T_{A/X}$. Indeed, we prove the following result:
\begin{theorem}\label{thm:intro2}
Let $A$ be eventually coconnective and let $X$ be a formal moduli problem over $A$ with associated Lie algebroid $T_{A/X}$. Then there is a fully faithful, symmetric monoidal left adjoint functor
$$\xymatrix{
\Psi_X\colon \QC(X)\ar[r] & \Rep_{T_{A/X}}
}$$
which induces an equivalence on connective objects.
\end{theorem}
It is well-known (cf.\ \cite{sai89, kap91}) that there are convergence issues that prevent the functor $\Psi_X$ from being an equivalence. For example, for any map $f\colon X\rt Y$ of formal moduli problems, the restriction functor 
$$\xymatrix{
f^!\colon \Rep_{T_{A/Y}}\ar[r] & \Rep_{T_{A/X}}
}$$
preserves all colimits, while the restriction functor $f^*\colon \QC(Y)\rt \QC(X)$ does not. For this reason (among many others), a more refined geometric theory of sheaves has been introduced in \cite{gai13}, in which quasi-coherent sheaves are replaced by \emph{ind-coherent sheaves}. The work of Gaitsgory and Rozenblyum \cite{gai16} provides an extensive study of derived algebraic geometry in terms of ind-coherent sheaves. In particular, they \emph{define} and study Lie algebroids purely in terms of formal moduli problems. More precisely, they consider a slight enlargement of the $\infty$-category of small extension of $A$, which also contains small extensions of $A$ by \emph{coherent} connective $A$-modules. A Lie algebroid is then defined to be a functor
$$\xymatrix{
X\colon \CAlg^\mm{sm, coh}_k/A\ar[r] & \sS
}$$
on this category satisfying conditions (a) and (b) of Definition \ref{def:formalmoduli}. We will refer to such an object as a \emph{pro-coherent formal moduli problem}.

In certain (somewhat restricted) situations, we show that the rectification results from Theorem \ref{thm:intro1} and Theorem \ref{thm:intro2} can also be applied to such pro-coherent formal moduli problems. In particular, this shows that pro-coherent formal moduli problems can indeed be interpreted at the point-set level as Lie algebroids. To obtain a simple point-set model for Lie algebroids in the ind-coherent setting, it will be more convenient to work with the $\infty$-category $\mm{Ind}(\cat{Coh}_A^\op)$ of \emph{pro-coherent sheaves} on $A$. This $\infty$-category is often equivalent to the $\infty$-category of ind-coherent sheaves by Serre duality \cite[Section 9]{gai13}, but has the advantage of admitting an explicit presentation by a certain `tame' model structure on dg-$A$-modules \cite{bec14}.

This tame model structure can be transferred to a model structure on dg-Lie algebroids over $A$. Unfortunately, this only presents the correct $\infty$-category of Lie algebroids in $\mm{Ind}(\cat{Coh}_A^\op)$ when $A$ satisfies some rather strict technical assumptions at the chain level (see Warning \ref{war:boundecoh}). As an important example, these technical conditions are met when $k$ is a coherent and eventually coconnective and when
$$
\spec(A) = \field{A}_k^n\times_{\field{A}_k^m}\{0\}.
$$
In particular, for $A=k$ we obtain a simple point-set model for the $\infty$-category of pro-coherent Lie algebras over $A$.

In the cases where we have a good point-set model for Lie algebroids in $\mm{Ind}(\cat{Coh}_A^\op)$, we show (Theorem \ref{thm:formalmodulicoherent}) that there is an equivalence
$$\xymatrix{
T_{A/}\colon \cat{FormMod}_A^!\ar[r]^-\sim & \LieAlgd_A^!
}$$
between the $\infty$-category of pro-coherent formal moduli problems over $A$ and the $\infty$-category of pro-coherent Lie algebroids over $A$. Theorem \ref{thm:intro2} can then be refined (Theorem \ref{fm:thm:ind-coherentvsrepresentations}) to an \emph{equivalence}
$$\xymatrix{
\Psi_X\colon \QC^!(X)\ar[r]^-\sim & \Rep_{T_{A/X}}^!
}$$
between pro-coherent sheaves on $X$ and pro-coherent $T_{A/X}$-representations.

\subsection*{Outline}
The paper is outlined as follows. In Section \ref{sec:dglie}, we recall the basic homotopy theory of dg-Lie algebroids over a commutative dg-algebra $A$, in which the weak equivalences are the quasi-isomorphisms. In Section \ref{sec:procoherent} we provide model categorical descriptions of the homotopy theories of pro-coherent sheaves and pro-coherent Lie algebroids over $A$. Theorem \ref{thm:intro1} and its analogue in the coherent case are proven in Section \ref{sec:formmod}, based on results about Lie algebroid cohomology that are discussed in Section \ref{sec:ce} and \ref{sec:koszulduality}.

Section \ref{sec:representations} is devoted to a proof of Theorem \ref{thm:intro2}. As an application, we also show how Theorem \ref{thm:intro2} can be used to provide a simple point-set model for the Lie algebroid classifying the deformations of a (connective) commutative algebra over $A$ (Proposition \ref{prop:defofalgebras}). Finally, in Section \ref{fm:sec:indcoh} we prove the equivalence of pro-coherent sheaves on formal moduli problems and representations of their associated Lie algebroid.

\subsection*{Conventions}
Throughout, let $\field{Q}\subseteq k$ be a fixed connective commutative dg-algebra of characteristic zero and let $A$ be a connective cdga over $k$. 
All differential-graded objects are homologically graded, so that connective objects are concentrated in non-negative degrees. Given a chain complex $V$, we denote its suspension and cone by $V[1]$ and $V[0, 1]$. 

\subsection*{Acknowledgement}
The author was supported by NWO.

\section{Recollections on DG-Lie algebroids}\label{sec:dglie}
In this section we recall the homotopy theory of dg-Lie algebroids over a commutative dg-algebra, based on the discussion in \cite{nui17b}.

\subsection{DG-Lie algebroids}
Recall that the tangent module of a commutative dg-$k$-algebra $A$ is the dg-$A$-module of $k$-linear derivations of $A$
$$
T_A = \mm{Der}_k(A, A).
$$ 
The commutator bracket endows this complex with the structure of a dg-Lie-algebra over $k$.
\begin{definition}\label{la:def:lralgebra}
A \emph{dg-Lie algebroid} $\mf{g}$ over $A$ (relative to $k$) is an unbounded dg-$A$-module $\mf{g}$, equipped with a $k$-linear dg-Lie algebra structure and an \emph{anchor map} $\rho\colon \mf{g}\rt T_A$ such that
\begin{enumerate}
\item $\rho$ is both a map of dg-$A$-modules and dg-Lie algebras.
\item the failure of the Lie bracket to be $A$-bilinear is governed by the Leibniz rule
$$
[X, a\cdot Y] = (-1)^{Xa} a[X, Y] + \rho(X)(a)\cdot Y.
$$
\end{enumerate}
Let $\dgLieAlgd_A$ be the category of dg-Lie algebroids over $A$, with maps given by $A$-linear maps over $T_A$ that preserve the Lie bracket.
\end{definition}
\begin{example}\label{ex:atiyah}
Any dg-$A$-module $E$ gives rise to an \emph{Atiyah dg-Lie algebroid} $\mm{At}(E)$ over $A$, which can be described as follows: a degree $n$ element of $\mm{At}(E)$ is a tuple $(v, \nabla_v)$ consisting of a derivation $v\colon A\rt A$ (of degree $n$), together with a $k$-linear map $\nabla_v\colon E\rt E$ (of degree $n$) such that
$$
\nabla_v(a\cdot e) = v(a)\cdot e + (-1)^{|a|\cdot n} a\cdot \nabla_v(e)
$$
for all $a\in A$ and $e\in E$. This becomes a dg-$A$-module under pointwise multiplication and a dg-Lie algebra under the commutator bracket. The anchor map is the obvious projection $\mm{At}(E)\rt T_A$ sending $(v, \nabla_v)$ to $v$.
\end{example}
\begin{example}\label{ex:atiyah2}
Similarly, suppose that $E\in \dgMod_A$ has the structure of an algebra over a $k$-linear dg-operad $\ope{P}$. Then there is a sub dg-Lie algebroid $\mm{At}_{\ope{P}}(E)\subseteq \mm{At}(E)$ consisting of the tuples $(v, \nabla_v)$ where $\nabla_v$ is a $\ope{P}$-algebra derivation.
\end{example}
The following is the main result of \cite{nui17b}:
\begin{proposition}\label{prop:monadicity}
The category of dg-Lie algebroids over $A$ carries the \emph{projective} semi-model structure, in which a map is a weak equivalence (fibration) if and only if it is a quasi-isomorphism (degreewise surjective). Furthermore, the forgetful functor to the projective model structure on dg-$A$-modules
$$\smash{\xymatrix{
\dgLieAlgd_A\ar[r] & \dgMod_A/T_A
}}$$
is a right Quillen functor that preserves all sifted homotopy colimits.
\end{proposition}
\begin{remark}\label{rem:equivalenttoquillen}
Let $0\rt \mf{g}\rt T_A$ be a fibrant-cofibrant replacement of the initial dg-Lie algebroid over $A$. Then there is a Quillen equivalence
$$\xymatrix{
\dgLieAlgd_A\ar@<1ex>[r] & \mf{g}/\dgLieAlgd_A .\ar@<1ex>[l]
}$$
Using the simplicial structure from \cite{vez13} and the fact that every object in $\mf{g}/\dgLieAlgd_A$ is fibrant, one finds that $\mf{g}/\dgLieAlgd_A$ is a genuine (combinatorial) model category.
\end{remark}
\begin{definition}\label{def:liealgebroids}
Let $A$ be a cofibrant connective cdga over $k$. We define the \emph{$\infty$-category of Lie algebroids over $A$} to be the $\infty$-categorical localization
$$
\LieAlgd_A := \dgLieAlgd_A\Big[\{\text{quasi-iso}\}^{-1}\Big].
$$
This is a locally presentable $\infty$-category since $\dgLieAlgd_A$ is Quillen equivalent to a combinatorial model category (Remark \ref{rem:equivalenttoquillen}).
\end{definition}
\begin{remark}\label{rem:derivedtangent}
The condition that $A$ is cofibrant over $k$ guarantees that the tangent module $T_A$ has the correct homotopy type. 
\end{remark}
\begin{definition}\label{def:Q}
We will say that a dg-Lie algebroid $\mf{g}$ is \emph{$A$-cofibrant} when it is cofibrant as a dg-$A$-module. Every cofibrant dg-Lie algebroid is $A$-cofibrant. Conversely, if $\mf{g}$ is $A$-cofibrant, then it has an explicit cofibrant replacement $Q(\mf{g})\rto{\sim} \mf{g}$, described as follows \cite[Section 5]{nui17b}: without differential, $Q(\mf{g})$ is freely generated by the $A$-linear map
$$\xymatrix{
\Big(\mm{Sym}^{\sgeq 1}_A\mf{g}[1]\Big)[-1]\ar[r]^-{\pi} & \mf{g}\ar[r]^-\rho & T_A
}$$
where $\pi$ is the obvious projection. As a dg-Lie algebra, $Q(\mf{g})$ is a quotient of the $A$-linear extension usual \cite{lod12} operadic bar-cobar resolution $\Omega B\mf{g}$ of the dg-Lie algebra underlying $\mf{g}$; this determines the differential.
\end{definition}

\subsection{Representations}
Recall that a \emph{representation} of a dg-Lie algebroid $\mf{g}$ is a dg-$A$-module $E$, together with a Lie algebra representation $\nabla\colon \mf{g}\otimes_k E\rt E$ such that (without Koszul signs)
\begin{equation}\label{eq:reps}
\nabla_{a\cdot X}(e) = a\cdot \nabla_X(e) \qquad \qquad \nabla_X(a\cdot e) = a\cdot \nabla_X(e) + \rho(X)(a)\cdot e
\end{equation}
for all $a\in A, X\in\mf{g}$ and $e\in E$. Equivalently, a $\mf{g}$-representation on $E$ is a map
$$\xymatrix{
\mf{g}\ar[r] & \mm{At}(E)
}$$ 
to the Atiyah Lie algebroid of $E$ (Example \ref{ex:atiyah}). Such representations can be organized into a symmetric monoidal category $\dgRep_\mf{g}$, with tensor product given by $E\otimes_A F$, endowed with the $\mf{g}$-representation
$$
\nabla_X(e\otimes f) = \nabla_X(e)\otimes f + e\otimes \nabla_X(f).
$$
The internal hom is given by $\Hom_A(E, F)$, endowed with the $\mf{g}$-representation by conjugation.
\begin{example}\label{ex:algebrainrep}
Let $\ope{P}$ be a dg-operad. Then a $\ope{P}$-algebra in $\dgRep_\mf{g}$ is simply a $\ope{P}$-algebra in $\smash{\dgMod_A}$, equipped with a $\mf{g}$-representation acting by $\ope{P}$-algebra derivations. Equivalently, such a $\ope{P}$-algebra structure on a $\mf{g}$-representation $E$ is determined by a map of dg-Lie algebroids $\mf{g}\rt \mm{At}_{\ope{P}}(E)$ (Example \ref{ex:atiyah2}).
\end{example}
One can identify $\dgRep_\mf{g}$ with the category of left modules over the \emph{universal enveloping algebra} $\mc{U}(\mf{g})$ of $\mf{g}$. In particular, it carries a model structure transferred from the projective model structure on dg-$A$-modules.
\begin{lemma}\label{lem:pbw}
Any map between dg-Lie algebroids $f\colon \mf{g}\rt \mf{h}$ induces a Quillen adjunction
$$\xymatrix{
f_*\colon \dgRep_\mf{g}\ar@<1ex>[r] & \dgRep_\mf{h}\ar@<1ex>[l] \colon f^!
}$$
where $f^!$ restricts a representation along $f$. When $\mf{g}$ and $\mf{h}$ are $A$-cofibrant and $f$ is a weak equivalence, this Quillen adjunction is a Quillen equivalence.
\end{lemma}
\begin{proof}
The map $f$ induces a map of universal enveloping algebras $\mc{U}(\mf{g})\rt \mc{U}(\mf{h})$, which gives rise to the Quillen pair $(f_*, f^!)$. The left adjoint $f_*$ is given by the functor $E\mapsto \mc{U}(\mf{h})\otimes_{\mc{U}(\mf{g})} E$. 

For the second part, recall that the universal enveloping algebra $\mc{U}(\mf{g})$ has a natural PBW filtration, obtained by declaring generators from $A$ to have weight $0$ and generators from $\mf{g}$ to have weight $1$ \cite{rin63}. The associated graded is a graded cdga and comes equipped with a surjective map of graded cdgas
$$\xymatrix{
\mm{Sym}_A \mf{g}\ar[r] & \mm{gr}\big(\mc{U}(\mf{g})\big).
}$$
When $\mf{g}$ is projective as a graded $A$-module, this map (or rather, the underlying map of graded graded-commutative algebras) is an isomorphism by the PBW theorem of \cite{rin63}, which applies verbatim in the graded setting. 

Now suppose that $\mf{g}\rt \mf{h}$ is a weak equivalence between $A$-cofibrant dg-Lie algebroids. The map $\mc{U}(\mf{g})\rt \mc{U}(\mf{h})$ induces a quasi-isomorphism on the associated graded and is therefore a weak equivalence itself. This implies that $(f_*, f^!)$ is a Quillen equivalence.
\end{proof}
\begin{remark}
The right Quillen functor $f^!$ induces a right adjoint functor of $\infty$-categories $f^!\colon \Rep_\mf{h}\rt \Rep_\mf{g}$. This functor has a further right adjoint $f_!$, which sends $E$ to the derived coinduction $\Hom^h_{\mc{U}(\mf{g})}(\mc{U}(\mf{h}), E)$.
\end{remark}
\begin{remark}\label{rem:monoidalstructureonreps}
The tensor product $\otimes_A$ does not make $\dgRep_\mf{g}$ a monoidal model category. However, the tensor product $E\otimes_A (-)$ does preserve all weak equivalences whenever $E$ is cofibrant as a dg-$A$-module. When $\mf{g}$ is $A$-cofibrant, every cofibrant left $\mc{U}(\mf{g})$-module is cofibrant as an $A$-module by the PBW theorem, so that the tensor product $\otimes_A$ descends to a tensor product on the $\infty$-category $\Rep_\mf{g}$ \cite[Proposition 4.1.7.4]{lur16}. Because the forgetful functor $\Rep_\mf{g}\rt \Mod_A$ detects equivalences and preserves all (co)limits, the resulting symmetric monoidal structure on $\Rep_\mf{g}$ is closed.
\end{remark}
By Lemma \ref{lem:pbw} and Remark \ref{rem:monoidalstructureonreps}, any map of $f\colon \mf{g}\rt \mf{h}$ of dg-Lie algebroids induces a symmetric monoidal left and right adjoint $f^!\colon \Rep_\mf{h}\rt \Rep_\mf{g}$ between presentable (closed) symmetric monoidal $\infty$-categories. It follows from Lemma \ref{lem:pbw} that this determines a functor
$$\xymatrix{
\Rep\colon \LieAlgd_A^\op\ar[r] & \cat{Pr^\mm{L}_{sym. mon.}}\big/\Mod_A := \CAlg(\cat{Pr^L})/\Mod_A
}$$
to the $\infty$-category of presentable (closed) symmetric monoidal $\infty$-categories over $\Mod_A$, which sends $\mf{g}$ to the forgetful functor $\Rep_\mf{g}\rt \Mod_A$. 
\begin{lemma}\label{fm:lem:algrepspreservessiftedcolim}
The functor $\Rep$ preserves all limits.
\end{lemma}
\begin{proof}
We can forget about the symmetric monoidal structures, since the forgetful functor $\cat{Pr^\mm{L}_{sym. mon.}}=\CAlg(\cat{Pr^L})\rt \cat{Pr^\mm{L}}$ preserves limits and detects equivalences. For any map $f\colon \mf{g}\rt \mf{h}$, the restriction functor $f^!$ is both a left and a right adjoint. It therefore suffices to show that the functor $\Rep\colon \LieAlgd_A^\op\rt \cat{Pr^R}/\Mod_A$ preserves limits. 

To see this, consider $\mc{U}(\mf{g})$ as a unital associative algebra in the category of dg-$A$-\emph{bimodules} (over $\field{Q}$), using the map of algebras $A\rt \mc{U}(\mf{g})$. The category of dg-$A$-bimodules carries a monoidal model structure, whose weak equivalences and fibrations are transferred along the forgetful functor $\cat{BiMod}^\dg_A\rt \cat{LMod}^\dg_A$. Since we are working in characteristic zero, there exists a transferred model structure on the category of associative algebras in $\cat{BiMod}^\dg_A$ as well. By \cite[Theorem 4.1.8.4]{lur16}, this is a model for the $\infty$-category of algebras in the monoidal $\infty$-category $\cat{BiMod}_A$ of $A$-bimodules (in loc.\ cit., the monoid axiom and symmetry of the monoidal structure are only used for the existence of a model structure on algebras).

The functor $\Rep$ now factors over the $\infty$-category of algebras in $A$-bimodules
$$\xymatrix@C=3pc{
\LieAlgd_A^\op\ar[r]^-{\mc{U}} & \cat{Alg}\big(\cat{BiMod}_A\big)^\op \ar[r]^-{\mm{LMod}} & \cat{Pr^{R}}/\Mod_A.
}$$
The functor $\mm{LMod}$ preserves all limits by \cite[Theorem 4.8.5.11]{lur16} (see the remarks just above Corollary 4.8.5.13). It remains to verify that $\mc{U}$ preserves limits. To see that it preserves sifted limits, it suffices to check that the composite
$$\xymatrix{
\LieAlgd_A^\op\ar[r]^-{\mc{U}} & \cat{Alg}\big(\cat{BiMod}_A\big)^\op\ar[r]^-{\mm{forget}} & \cat{BiMod}^\op_A
}$$
preserves sifted limits. This was proven in \cite[Theorem 4.19]{nui17b}.

To see that $\mc{U}$ preserves finite products, note that $\LieAlgd_A$ is generated under sifted colimits by free Lie algebroids $F(V)$ on maps $0\colon V\rt T_A$ of $A$-modules. It therefore suffices to verify that the composite
$$\xymatrix{
\Mod_A\ar[r]^-{F} & \LieAlgd_A\ar[r]^-{\mc{U}} & \cat{Alg}\big(\cat{BiMod}_A\big)
}$$
preserves finite coproducts. Unraveling the definitions, one sees that for any cofibrant dg-$A$-module $F$, the dg-algebra $\mc{U}(F(V))$ is naturally equivalent to the $A$-linear tensor algebra $T_A(V)$. In other words, the above functor is naturally equivalent to the functor
$$\xymatrix{
\Mod_A\ar[r]^\Delta & \cat{BiMod}_A\ar[r]^-{\mm{Free}} & \cat{Alg}\big(\cat{BiMod}_A\big)
}$$
sending an $A$-module to the free algebra on $V$, considered as a symmetric $A$-bimodule. This functor clearly preserves colimits.
\end{proof}

\section{Pro-coherent sheaves and Lie algebroids}\label{sec:procoherent}
The $\infty$-category of Lie algebroids over $A$ defined in the previous section was obtained by formally inverting the quasi-isomorphisms between dg-Lie algebroids over $A$. In this section, we discuss the \emph{tame} homotopy theory of dg-Lie algebroids, which has \emph{less} weak equivalences and is transferred from the tame model structure on dg-$A$-modules. We will start by recalling the tame model structure on dg-$A$-modules in Section \ref{sec:tamemodules} and discuss the transferred homotopy theory on dg-Lie algebroids in Section \ref{sec:tameliealgebroids}. 

Our main aim will be to show that in suitable cases, the tame model structure on dg-$A$-modules presents the $\infty$-category $\mm{Ind}(\cat{Coh}_A^\op)$ of \emph{pro-coherent sheaves} on $A$ (or rather, its opposite). The tame homotopy theory theory on dg-Lie algebroids can then be considered as a model for the $\infty$-category of Lie algebroids in pro-coherent sheaves over $A$. In particular, the $\infty$-category of such pro-coherent Lie algebroids is \emph{compactly generated}, which plays an important technical role in our discussion of Koszul duality. To make sure that the tame model structure on dg-$A$-modules indeed presents pro-coherent sheaves, we need to put some restrictions on $A$:
\begin{definition}
A connective commutative dg-$k$-algebra $A$ is called
\begin{enumerate}
\item (strictly) \emph{bounded} if it is concentrated in degrees $[0, n]$, for some $n\geq 0$.
\item \emph{coherent} if $\pi_0(A)$ is a coherent ring and each $\pi_n(A)$ is a finitely presented $\pi_0(A)$-module.
\end{enumerate}
\end{definition}
\begin{warning}\label{war:boundecoh}
The tame homotopy theory of dg-Lie algebroids described below is only well-behaved when $A$ is both cofibrant over $k$ (Remark \ref{rem:derivedtangent}) and bounded coherent (Proposition \ref{prop:tamecompactlygenerated}). There is some conflict between cofibrancy and being bounded: for example, $A$ cannot have any generators of degree $2$.
\end{warning}
\begin{example}
Suppose that $k$ is bounded coherent and that $A$ is obtained by adding finitely many generators in degrees $0$ and $1$. Then $A$ is cofibrant over $k$ and bounded coherent. Geometrically, $A$ describes the derived zero locus of a map $\field{A}_k^n\rt \field{A}_k^m$ over $\spec(k)$.
\end{example}

\subsection{Tame dg-$A$-modules and pro-coherent sheaves}\label{sec:tamemodules}
Recall that for any commutative dg-algebra $A$, the category of dg-$A$-modules can be endowed with the \emph{tame}, or \emph{contraderived} model structure on dg-$A$-modules (see e.g.\ \cite{bec14}). In this model structure, a map of dg-$A$-modules $E\rt F$ is
\begin{itemize}[leftmargin=*]
\item[-] a fibration if it is degreewise surjective.
\item[-] a cofibration if it is a monomorphism, whose cokernel is projective as a graded $A$-module.
\item[-] a weak equivalence if for any graded-projective dg-$A$-module $P$, the map on hom-complexes $\Hom_A(P, E)\rt \Hom_A(P, F)$ is a quasi-isomorphism.
\end{itemize}
We will denote the associated $\infty$-category by $\Mod_A^!$ and refer to it as the $\infty$-category of \emph{tame dg-$A$-modules}. This terminology is supposed to emphasize that $\Mod_A^!$ depends on the explicit cdga $A$, i.e.\ quasi-isomorphic cdgas may have non-equivalent $\infty$-categories of tame dg-$A$-modules.

The tame model structure is stable and symmetric monoidal for the usual tensor product over $A$, and the usual projective model structure is a (symmetric monoidal) right Bousfield localization of the tame model structure. It follows that there is a fully faithful, symmetric monoidal left adjoint functor $$\xymatrix{
\Mod_A\ar[r] & \Mod_A^!.
}$$
\begin{lemma}\label{lem:tamecofgenerated}
Let $\cat{T}^{\sleq n}$ be the set of graded-free dg-$A$-modules $T$ satisfying the following two conditions:
\begin{itemize}
 \item[(i)] $T$ has no generators in (homological) degree $>n$
 \item[(ii)] $T$ has finitely many generators in each individual degree.
\end{itemize}
Let $\cat{T}=\bigcup_n \cat{T}^{\sleq n}$. Then the collection of cone inclusions $\big\{T\rt T[0,1] : T\in\cat{T}\big\}$ is a set of generating cofibrations for the tame model structure. In particular, the tame model structure is cofibrantly generated.
\end{lemma}
\begin{proof}
A map $E\rt F$ has the right lifting property against all $T\rt T[0,1]$ if and only if the map $\Hom_A(T, E)\rt \Hom_A(T, F)$ is a trivial fibration for all $T\in \cat{T}$. Let $\cat{C}\subseteq \dgMod_A$ be the full subcategory of graded-projective dg-$A$-modules $P$ for which $\Hom_A(P, E)\rt \Hom_A(P, F)$ is a trivial fibration. We have to show that $\cat{C}$ contains all graded-projective dg-$A$-modules. To see this, let us make the following observations:
\begin{enumerate}[leftmargin=*]
 \item $\cat{C}$ is closed under retracts.
 \item Let $P\rt Q$ be a cofibration such that $P$ and $Q/P$ are contained in $\cat{C}$. Then $Q$ is contained in $\cat{C}$.
 \item If $\{P_\alpha\}$ is a (transfinite) sequence of cofibrations between objects in $\cat{C}$, then the colimit $\colim P_\alpha$ is contained in $\cat{C}$ as well.
\end{enumerate}
By (1), it suffices to verify that $\cat{C}$ contains all graded-free dg-$A$-modules. Suppose that $Q$ is a graded-free dg-$A$-module, with set of generators $S=\{y_i\}$. Let $\mc{B}\subseteq P(S)$ be the poset of subsets $S'\subseteq S$ which generate a sub dg-$A$-module of $Q$ that is contained in $\cat{C}$. It follows from (3) that any chain in $\mc{B}$ has an upper bound. By Zorn's lemma, $\mc{B}$ admits a maximal element $S_0$. Let $P=A\cdot S_0$ and consider the quotient $Q/P$. We claim that $Q/P$ is trivial, so that $Q\in \cat{C}$. 

Indeed, $Q/P$ has a set of generators $S-S_0$. If there exists a generator $y\in S\setminus S_0$ of degree $n$, one can (inductively) find a set $S'\subseteq S-S_0$ containing $y$ such that
\begin{itemize}[leftmargin=*]
\item[-] $A\cdot S'$ is closed under the differential.
\item[-] $S'$ contains finitely many generators in each degree $\leq n$ and no generators in degrees $>n$.
\end{itemize}
In particular, such $A\cdot S'\subseteq Q/P$ is contained in $\cat{T}$. Using (2), it follows that $A\cdot (S_0\cup S')\subseteq Q$ is contained in $\cat{C}$ as well. This contradicts maximality of $S_0$, so we conclude that $S-S_0$ is empty and $Q/P=0$.
\end{proof}
\begin{remark}\label{rem:connectivemodules}
Let $E$ be a dg-$A$-module and consider its connective cover
$$
\tau_{\sgeq 0} E= (\dots\rt E_1\rt Z_0(E)\rt 0)
$$
Then the map $\tau_{\sgeq 0} E\rt E$ is a tame weak equivalence if and only if $[T, E]=0$ for any $T\in \cat{T}^{\sleq -1}$. We will say that $E$ is a \emph{connective} tame dg-$A$-module if it satisfies these equivalent conditions. The full subcategory $\Mod_{A}^{!, \sgeq 0}$ on the connective tame dg-$A$-modules determines an accessible $t$-structure on $\Mod_A^!$.

One can easily verify that a map between connective tame dg-$A$-modules is a tame weak equivalence if and only if it is a quasi-isomorphism. It follows that the fully faithful inclusion $\Mod_A\rt \Mod_A^!$ induces an equivalence $\Mod_A^{\sgeq 0}\rt \Mod_A^{!,\sgeq 0}$ on connective objects.
\end{remark}
The remainder of this section is devoted to a proof of the following result, proven in e.g.\ \cite{nee08} in the case where $A$ is discrete:
\begin{proposition}\label{prop:tamecompactlygenerated}
Suppose that $A$ is bounded and coherent and let $\mc{K}^{\leq 0}$ be the set of (weak equivalence classes of) dg-$A$-modules $E$ such that
\begin{itemize}
\item[(i)] they are in $\cat{T}^{\sleq 0}$, i.e.\ graded-free, with generators of degree $\leq 0$ and finitely many generators in each degree.
\item[(ii)] the $A$-linear dual $E^\vee = \Hom_A(E, A)$ is eventually coconnective, i.e.\ $\pi_i(E^\vee)$ vanishes for $i\gg 0$.
\end{itemize}
Then $\cat{K}^{\leq 0}$ is a set of compact generators for $\Mod_A^!$.
\end{proposition}
Before turning to the proof of Proposition \ref{prop:tamecompactlygenerated}, let us record the following consequence. Recall that for coherent $A$, an object $E\in \Mod_A$ is a \emph{coherent} $A$-module if it has finitely many nontrivial homotopy groups $\pi_i(E)$, each of which is finitely presented over $\pi_0(A)$. Equivalently \cite[Proposition 7.2.4.17]{lur16}, it is eventually coconnective and almost finitely presented, i.e.\ presentable by a graded-free dg-$A$-module with finitely many generators in each degree and no generators in degrees $\ll 0$.
\begin{corollary}\label{cor:tameasprocoh}
Let $A$ be bounded and coherent and consider the functor 
$$\smash{\xymatrix{
(-)^\vee\colon \Mod_A^!\ar[r] & \Mod_A^\op
}}$$ 
sending a tame dg-$A$-module to its $A$-linear dual. This functor restricts to an equivalence between the full subcategory $\cat{K}^{\leq 0}\subseteq \Mod_A^!$ and the opposite of the full subcategory $\cat{Coh}^{\sgeq 0}_A\subseteq \Mod_A$ of connective coherent $A$-modules. In particular, there is an equivalence of stable $\infty$-categories
$$
\Mod_A^!\simeq \mm{Ind}(\cat{Coh}_A^\op).
$$
\end{corollary}
\begin{proof}
The functor $(-)^\vee\colon \dgMod_A\rt \Mod_A^{\dg, \op}$ is a left Quillen functor from the tame model structure to the projective model structure, whose right adjoint is the functor $(-)^\vee$ as well. For any (cofibrant) object $E\in\mc{K}^{\leq 0}$, the dual $E^\vee$ is graded-free, with finitely many generators in each nonnegative degree. In particular, it is cofibrant. Since $A$ is bounded, the derived unit map $E\rt E^{\vee\vee}$ is an isomorphism. The full subcategory of $\Mod_A^!$ on the compact generators is therefore equivalent to its essential image under $(-)^\vee$ in $\Mod_A^\op$. Unwinding the definitions, this essential image is exactly the opposite of $\cat{Coh}_A^{\sgeq 0}$.
\end{proof}
\begin{definition}
When $A$ is bounded coherent, we will refer to $\Mod_A^!=\mm{Ind}(\cat{Coh}_A^\op)$ as the $\infty$-category of \emph{pro-coherent sheaves} on $A$, even though this is technically its opposite category.
\end{definition}
\begin{example}
Let $f\colon A\rt B$ be a weak equivalence of bounded coherent cdgas. Then there is a Quillen equivalence $f^!\colon \dgMod_A\leftrightarrows \dgMod_B\colon f_!$ between the tame model structures, where $f^!$ sends $E$ to $B\otimes_A E$. Indeed, for every coherent $A$-module $E$, we have that $f^!(E)$ is a coherent $B$-module. Because $f^!(E^\vee) \cong f^!(E)^\vee$, the functor $f^!$ preserves compact objects and the composite
$$\xymatrix{
\cat{Coh}_A^\op\ar[r]^{(-)^\vee} & \Mod_A^{!, \omega}\ar[r]^{f^!} & \Mod_B^{!, \omega}\ar[r]^{(-)^\vee} & \cat{Coh}_B^\op
}$$
is an equivalence. This implies that $f^!\colon \Mod_A^!\rt \Mod_B^!$ is an equivalence.
\end{example}
\begin{lemma}\label{lem:gradedfree}
Let $E$ be an object in $\mc{K}^{\sleq 0}$ and let $n_0$ be an integer such $\pi_i(E^\vee)=0$ for all $i\geq n_0$. Then the following holds:
\begin{enumerate}
\item For any indexing set $I$, the map $\bigoplus_I\Hom_A(E, A)\rt \Hom_A(E, \bigoplus_I A)$ is an isomorphism.
\item Let $F$ be a dg-$A$-module which is graded-free, with generators $x_i$. Let $F^{(n)}$ be the quotient of $F$ by the sub dg-$A$-module generated by the $x_i$ of degree $<n$. Then for all $n>1$:
$$
\pi_0\Hom_A(E, F^{(n)})=0.
$$
\item For every $m\leq -n_0$, the map
$$\xymatrix{
\pi_0\Hom_A(E, F)\ar[r] & \pi_0\Hom_A(E, F^{(m)})
}$$
is an isomorphism.
\end{enumerate}
\end{lemma}
\begin{proof}
For (1), we can write $E=\bigoplus_{n\sgeq 0} A^{\oplus i_n}[-n]$ at the level of graded $A$-modules. The map $\bigoplus_I\Hom_A(E, A)\rt \Hom_A(E, \bigoplus_I A)$ is then given by
$$\xymatrix{
\bigoplus_I \prod_{n\geq 0} A^{\oplus i_n}[n]\ar[r] & \prod_{n\geq 0} \bigoplus_I A^{\oplus i_n}[n].
}$$
This map is an isomophism in each degree, since $A$ is concentrated in nonnegative degrees.

For assertion (2), note that $F^{(n)}$ is graded-free on the generators $x_i$ of degree $\geq n$. Since $E$ has generators in degree $\leq 0$, the complex $\Hom_A(E, F^{(n)})$ is trivial in degree $0$ when $n>1$. 

For (3), consider the tower of fibrations between graded-free dg-$A$-modules
$$\xymatrix{
F\ar[r] & \dots \ar[r] & F^{(-2)}\ar[r] & F^{(-1)}\ar[r] & F^{(0)}.
}$$
The natural map $F\rt \lim_{n\leq 0} F^{(n)}$ to the (homotopy) limit of this tower can be identified with the map
$$\xymatrix{
\bigoplus_{n} \bigoplus_{I_n} A[-n]\ar[r] & \prod_{n\leq 0} \bigoplus_{I_n} A[-n].
}$$
Because $A$ is concentrated in degrees $[0, n]$, this is an isomorphism.

For every $m\leq -n_0$, the fiber of the map $\Hom_A(E, F^{(m-1)})\rt \Hom_A(E, F^{(m)})$ can be identified with
$$
\Hom_A\Big(E, \bigoplus_{I_m} A[m-1]\Big)\cong \bigoplus_{I_m} \Hom_A\big(E, A[m-1]\big).
$$
By assumption, the homotopy groups of this complex vanish in all degrees $\geq -1$. This implies that the map 
$$\xymatrix{
\pi_0\Hom_A(E, F)\cong \pi_0\lim_{n\leq 0}\Hom_A(E, F^{(n)})\ar[r] & \pi_0\Hom_A(E, F^{(m)})
}$$
is an isomorphism.
\end{proof}
\begin{lemma}\label{lem:generate}
Let $\cat{C}\subseteq \Mod_A^!$ be the smallest full subcategory of $\Mod_A^!$ which contains the objects in $\mc{K}^{\sleq 0}$ and is closed under colimits and extensions. Then $\cat{C}=\Mod_A^!$.
\end{lemma}
\begin{proof}
By Lemma \ref{lem:tamecofgenerated}, it suffices to show that $\cat{C}$ contains all objects $T\in \cat{T}$, which are of the form $\bigoplus_{n\geq n_0} A^{\oplus k_n}[-n]$ without differential. The dual of such a $T$ is given by the (projecively) cofibrant dg-$A$-module $\bigoplus_{n\geq n_0} A^{\oplus k_n}[n]$ and the map $T\rt T^{\vee\vee}$ is an isomorphism. Consider the Postnikov tower of $T^\vee$
$$\xymatrix{
T^\vee\ar[r] & \dots\ar[r] & \tau_{\sleq n}(T^\vee)\ar[r] & \tau_{\sleq n-1}(T^\vee)\ar[r] & \cdots \ar[r] & \pi_0(T^\vee).
}$$
Because $A$ is coherent, each $\pi_i(T^\vee)$ is a coherent $A$-module and admits an almost finite resolution $Y_i\rto{\sim} \pi_i(T^\vee)$ \cite[Proposition 7.2.4.17]{lur16}. One can use these resolutions to resolve the entire Postnikov tower by
$$\smash{\xymatrix{
T^\vee\ar[r] & \dots \ar[r] & P_n\ar[r] & P_{n-1}\ar[r] & \dots \ar[r] & P_0
}}$$
where $P_n=\bigoplus_{i=0}^n Y_i[i]$, equipped with a certain differential. The sequence of $P_n$ becomes stationary in each individual degree, so that there is a homotopy equivalence $T^\vee\rt P_\infty = \bigoplus_{i} Y_i[i]$. Taking the dual, one finds that $T=T^{\vee\vee}$ is homotopy equivalent to the colimit of the sequence
$$\smash{\xymatrix{
P_0^\vee\ar[r] & P_1^\vee\ar[r] & \dots.
}}$$
This is a sequence of cofibrations whose associated graded consists of the $Y_i[i]^\vee$. Each $Y_i[i]^\vee$ is contained in $\mc{K}$, because its dual is $Y_i[i]\simeq \pi_i(T^\vee)[i]$. This implies that the (homotopy) colimit $T$ is contained in the category $\cat{C}$.
\end{proof}
\begin{proof}[Proof (of Proposition \ref{prop:tamecompactlygenerated})]
By Lemma \ref{lem:generate}, the objects of $\mc{K}^{\sleq 0}$ generate $\Mod_A^!$. To prove that any $E\in \mc{K}^{\sleq 0}$ is compact, it suffices to show that for any set $S$ and any collection of graded-free dg-$A$-modules $\{P_\alpha\}_{\alpha\in S}$, the map
$$\xymatrix{
\bigoplus_\alpha \pi_0\Hom_A(E, P_\alpha)\ar[r] & \pi_0\Hom_A\Big(E, \bigoplus_\alpha P_\alpha\Big)
}$$
is an isomorphism. Using the filtration of Lemma \ref{lem:gradedfree}, we obtain commuting squares
$$\xymatrix@R=1.7pc{
\bigoplus_\alpha \pi_0\Hom_A(E, P_\alpha)\ar[d]\ar[r] & \pi_0\Hom_A\Big(E, \bigoplus_\alpha P_\alpha\Big)\ar[d]\\
\bigoplus_\alpha \pi_0\Hom_A(E, P^{(n)}_\alpha)\ar[r]^{\phi_n} & \pi_0\Hom_A\Big(E, \bigoplus_\alpha P^{(n)}_\alpha\Big)
}$$
for all $n\in \field{Z}$. We will prove by decreasing induction on $n$ that $\phi_n$ is an isomorphism; this proves that the top map is an isomorphism, because the vertical maps become isomorphisms for all $m$ smaller than a certain $n_0$, by Lemma \ref{lem:gradedfree}. We can start the induction at $n=2$, where both objects are zero by Lemma \ref{lem:gradedfree}.

For the inductive step, suppose that $\phi_n$ is an isomorphism. To prove that $\phi_{n-1}$ is an isomorphism as well, let $F_\alpha$ denote the fiber of the map $P_\alpha^{(n-1)}\rt P_\alpha^{(n)}$. It suffices to check that the map
$$\smash{\xymatrix{
\bigoplus_\alpha \Hom_A(E, F_\alpha)\ar[r] & \Hom_A\Big(E, \bigoplus F_\alpha\Big)
 }}$$ 
is a quasi-isomorphism. But each $F_\alpha$ is just given by a direct sum $\bigoplus A[n-1]$, so the result follows from (a shift of) Lemma \ref{lem:gradedfree}.
\end{proof}

\subsection{Tame dg-Lie algebroids}\label{sec:tameliealgebroids}
The tame model structure on dg-$A$-modules can be transferred to a semi-model structure on dg-Lie algebroids, by \cite[Remark 4.25]{nui17b}:
\begin{proposition}\label{prop:monadicitytame}
The category of dg-Lie algebroids over $A$ carries the \emph{tame} semi-model structure, in which a map is a weak equivalence (fibration) if and only if it is a tame weak equivalence (fibration). The forgetful functor $\dgLieAlgd_A\rt \dgMod_A/T_A$ is a right Quillen functor to the tame model structure, which preserves all sifted homotopy colimits.
\end{proposition}
\begin{definition}\label{def:tameliealgebroids}
Let $A$ be a cofibrant connective cdga over $k$. The $\infty$-category of \emph{tame} dg-Lie algebroids over $A$ is the $\infty$-categorical localization
$$
\LieAlgd^!_A := \dgLieAlgd_A\Big[\{\text{tame w.e.}\}^{-1}\Big].
$$
This is a locally presentable $\infty$-category by Remark \ref{rem:equivalenttoquillen}. When $A$ is bounded coherent, one can also think of $\LieAlgd_A^!$ as the $\infty$-category of \emph{pro-coherent Lie algebroids} over $A$.
\end{definition}
The projective model semi-model structure on dg-Lie algebroids is a right Bousfield localization of the tame semi-model structure, so that there is a fully faithful left adjoint functor
$$\xymatrix{
\LieAlgd_A\ar@{^{(}->}[r] & \LieAlgd_A^!.
}$$
The discussion of the previous sections carries over verbatim to the tame case. For example, every (tame) $A$-cofibrant dg-Lie algebroid has an explicit cofibrant replacement $Q(\mf{g})$, which is graded-free on $\big(\mm{Sym}_A\mf{g}[1]\big)[-1]$. Furthermore, we have the following:
\begin{lemma}\label{lem:propertiesoftamereps}
The category $\Rep_\mf{g}^\dg$ carries a model structure transferred from the tame model structure on $\dgMod_A$, whose associated $\infty$-category will be denoted $\Rep_\mf{g}^!$. This has the following properties:
\begin{enumerate}[leftmargin=*]
\item For every $f\colon \mf{g}\rt \mf{h}$, restriction and induction form a Quillen pair between tame model structures
$$\xymatrix{
f_*=\mc{U}(\mf{g})\otimes_{\mc{U}(\mf{h})} -\colon \Rep_\mf{g}^\dg\ar@<1ex>[r] & \dgRep_\mf{h}\ar@<1ex>[l]\colon f^!
}$$
which is a Quillen equivalence when $f$ is a tame weak equivalence between (tame) $A$-cofibrant dg-Lie algebroids.
\item $\Rep_A^!$ is a closed symmetric monoidal category.
\item The induced functor of $\infty$-categories
$$\xymatrix{
\Rep^!\colon \LieAlgd_A^{!, \op}\ar[r] & \cat{Pr^\mm{L}_{sym. mon.}}/\Mod_A^!
}$$
preserves all limits.
\end{enumerate}
\end{lemma}
\begin{proof}
The PBW filtration also applies to dg-Lie algebroids which are $A$-cofibrant for the tame model structure on dg-$A$-modules. The proof of Lemma \ref{lem:pbw} can therefore be applied to show (1). Assertion (2) follows from Remark \ref{rem:monoidalstructureonreps}.

For (3), one applies the same proof as in Lemma \ref{fm:lem:algrepspreservessiftedcolim}. Indeed, let $\cat{BiMod}^!_A$ be the $\infty$-category associated to the category of dg-$A$-bimodules (over $\field{Q}$), endowed with the model structure transferred from the \emph{tame} model structure on left dg-$A$-modules. This is a cofibrantly generated monoidal model category, with generating cofibrations with cofibrant domains. By \cite[Theorem 4.1.8.4]{lur16}, the $\infty$-category $\cat{Alg}(\cat{BiMod}^!_A)$ can be modeled by the transferred model structure on associative algebras in this monoidal model category. Forgetting the symmetric monoidal structure, the functor $\Rep$ now decomposes as
$$\xymatrix@C=3pc{
\LieAlgd_A^\op\ar[r]^-{\mc{U}} & \cat{Alg}\big(\cat{BiMod}^!_A\big)^\op \ar[r]^-{\mm{LMod}} & \cat{Pr^{R}}/\Mod_A.
}$$
The second functor preserves limits by \cite[Theorem 4.8.5.11]{lur16}. To see that the first functor preserves sifted limits, it suffices to show that $\mc{U}\colon \LieAlgd_A\rt \cat{BiMod}^!_A$ preserves sifted colimits. This follows from \cite[Theorem 4.19, Remark 4.23]{nui17b}. As in the proof of Lemma \ref{fm:lem:algrepspreservessiftedcolim}, the fact that $\mc{U}$ also preserves finite products follows from the fact that $\mc{U}(F(V))\cong T_A(V)$ for any cofibrant dg-$A$-module.
\end{proof}

\section{Lie algebroid cohomology}\label{sec:ce}
The purpose of this section is to prove the following:
\begin{proposition}\label{prop:deformationadjunction}
Let $A$ be a cofibrant commutative dg-$k$-algebra. There is an adjunction of $\infty$-categories
\begin{equation}\label{diag:deformationadjunction}\vcenter{\xymatrix{
C^*\colon \LieAlgd^!_A\ar@<1ex>[r] & \big(\CAlg_k/A\big)^\op\colon \mf{D}\ar@<1ex>[l]
}}\end{equation}
between the $\infty$-category of tame dg-Lie algebroids over $A$ and the $\infty$-category of unbounded commutative $k$-algebras over $A$ (i.e.\ cdgas up to quasi-isomorphism). The right adjoint sends $B\rt A$ to the dual in $\Mod_A^!$ of the map between cotangent complexes $L_A\rt L_{A/B}$ (over $k$).
\end{proposition}
\begin{remark}
Composing with the fully faithful left adjoint $\LieAlgd_A\rt \LieAlgd_A^!$, one obtains an adjunction $\smash{C^*\colon \LieAlgd_A\leftrightarrows \big(\CAlg_k/A\big)^\op\colon \mf{D}}$ as well. The right adjoint takes the dual of $L_A\rt L_{A/B}$ inside $\Mod_A\subseteq \Mod_A^!$.
\end{remark}
The left adjoint in \eqref{diag:deformationadjunction} is given by the \emph{Chevalley-Eilenberg complex} (with trivial coefficients). Recall that for any representation $E$ of a dg-Lie algebroid $\mf{g}$, this complex (with coefficients in $E$) is given by the graded vector space
$$
C^*(\mf{g}, E) := \Hom_A\Big(\mm{Sym}_A\mf{g}[1], E\Big)
$$
with the usual Chevalley-Eilenberg (or de Rham) differential (without Koszul signs)
\begin{align}\label{fm:eq:cediff}
 (\dau\alpha)(X_1, \dots, X_n) & =  \dau_E\big(\alpha(X_1, \dots, X_n)\big) - \sum_{i} \hspace{2pt} \alpha(X_1, \dots, \dau X_i, \dots, X_n) \nonumber\\
 &+ \sum_{i} \hspace{2pt} \nabla_{X_{i}}\Big(\alpha(X_1, \dots, X_n)\Big) - \sum_{i<j} \hspace{2pt} \alpha\Big([X_i, X_j], X_1, \dots, X_n\Big).
\end{align}
There is a natural $k$-linear augmentation map $C^*(\mf{g}, E)\rt E$, evaluating at the unit element of $\mm{Sym}_A\mf{g}[1]$. 
\begin{remark}\label{rem:celaxmonoidal}
The shuffle product of forms defines a lax symmetric monoidal structure on $C^*(\mf{g}, -)$, which is compatible with the augmentation in the sense that there is a commuting square
$$\xymatrix{
C^*(\mf{g}, E)\otimes_k C^*(\mf{g}, F)\ar[r]^-{\times}\ar[d] & C^*(\mf{g}, E\otimes_A F)\ar[d]\\
E\otimes_k F\ar[r] & E\otimes_A F.
}$$
In particular, taking coefficients with values in the commutative algebra $A$, one obtains a functor
$$\xymatrix{
C^*\colon \dgLieAlgd_A\ar[r] & \Big(\dgCAlg_k/A\Big)^\op; \hspace{4pt} \mf{g}\ar@{|->}[r] & C^*(\mf{g}):=C^*(\mf{g}, A).
}$$
For every dg-Lie algebroid $\mf{g}$, the Chevalley-Eilenberg complex yields a lax symmetric monoidal functor $C^*(\mf{g}, -)\colon \dgRep_\mf{g}\rt \dgMod_{C^*(\mf{g})}$.
\end{remark}
The proof of Proposition \ref{prop:deformationadjunction} is given at the very end of this section and uses some formal properties of the Chevalley-Eilenberg complex. To understand these properties, it will be useful to first give a slighty more model-categorical characterization of the Chevalley-Eilenberg complex.

\subsection{The cotangent complex of a Lie algebroid}
Consider the right Quillen functor 
$$\xymatrix{
\mf{g}\oplus (-)\colon \dgRep_\mf{g}\ar[r] & \dgLieAlgd_A/\mf{g}
}$$
sending a $\mf{g}$-representation $E$ to the square zero extension of $\mf{g}$ by $E$. If we denote the value of the left adjoint functor on $\mf{g}$ itself by $\Upsilon_\mf{g}$, then the other values of the left adjoint are given by
\begin{equation}\label{diag:kaehler}\xymatrix{
\Big(f\colon \mf{h}\rt \mf{g}\Big)\ar@{|->}[r] & f_*\Upsilon_\mf{h} = \mc{U}(\mf{g})\otimes_{\mc{U}(\mf{h})}\Upsilon_{\mf{h}}.
}\end{equation}
\begin{definition}\label{fm:def:cotangentcomplexofliealgebroid}
Let $\mf{g}$ be an $A$-cofibrant dg-Lie algebroid. The \emph{cotangent complex} $L_\mf{g}$ of $\mf{g}$ is the value of the left derived functor of \eqref{diag:kaehler} on the identity map of $\mf{g}$. In other words, it is the universal $\mf{g}$-representation classified by
$$
\Map_{\mc{U}(\mf{g})}(L_\mf{g}, E)\simeq \Map_{/\mf{g}}\big(\mf{g}, \mf{g}\oplus E\big).
$$
\end{definition}
\begin{example}\label{fm:ex:cotangentcomplexoffreealgebroid}
Let $\mf{g}=F(V)$ be the free dg-Lie algebroid generated by an $A$-linear map $V\rt T_A$ whose domain is cofibrant. Then $\mf{g}$ is a cofibrant dg-Lie algebroid and for any $\mf{g}$-representation $E$, there is a natural bijection between sections $\mf{g}\rt \mf{g}\oplus E$ and $A$-linear maps $V\rt E$. It follows that $L_\mf{g}=\mc{U}(\mf{g})\otimes_A V$ is the free $\mf{g}$-representation generated by the dg-$A$-module $V$.
\end{example}
More generally, one can use the cofibrant replacement $Q(\mf{g})$ of Definition \ref{def:Q} to compute the cotangent complex:
\begin{proposition}\label{prop:cotangentofliealgebroid}
Let $\mf{g}$ be an $A$-cofibrant dg-Lie algebroid. Then the cotangent complex $L_\mf{g}$ can be modeled by the cofibrant left $\mc{U}(\mf{g})$-module
\begin{equation}\label{eq:cotangentbycobar}
L_\mf{g} = \mc{U}(\mf{g})\otimes_A \left(\mm{Sym}_A^{\sgeq 1}\mf{g}[1]\right)[-1]
\end{equation}
with differential given (modulo Koszul signs) by
\begin{align}\label{fm:eq:koszuldifferential}
\dau(u\otimes X_1\dots X_n) &= (\dau u)\otimes X_1\dots X_n + \sum_i u\otimes X_1\dots \dau(X_i)\dots X_n \nonumber\\
& \smash{\stackrel{(n>1)}{+}} \sum_{i} u\cdot X_k\otimes X_{1}\dots X_{n} + \sum_{i<j} u \otimes [X_i, X_j]X_1\dots X_n.
\end{align}
The first term in the second row only applies when $n>1$.
\end{proposition}
\begin{proof}
Since $\mf{g}$ is $A$-cofibrant, it suffices to find a $\mf{g}$-representation corepresenting the functor
$$\xymatrix{
\dgRep_\mf{g}\ar[r] & \Set; \hspace{4pt} E\ar@{|->}[r] & \Hom_{/\mf{g}}(Q(\mf{g}), E).
}$$
By \cite[Corollary 6.13]{nui17b}, there is a natural bijection between maps $Q(\mf{g})\rt \mf{g}\oplus E$ over $\mf{g}$ and $0$-cycles in the kernel of the augmentation map $C^*(\mf{g}, E[1])\rt E[1]$. On the other hand, unwinding the definition of the complex $L_\mf{g}$ in \eqref{eq:cotangentbycobar}, one sees that there is a short exact sequence
\begin{equation}\label{eq:ceismapsfromcot}\vcenter{\xymatrix{
\Hom{\hspace{1pt}}_{\mc{U}(\mf{g})}\big(L_\mf{g}, E\big)\ar[r] & C^*(\mf{g}, E[1])\ar[r] & E[1].
}}\end{equation}
Passing to $0$-cycles, one finds that the complex $L_\mf{g}$ given in \eqref{eq:cotangentbycobar} indeed represents maps $Q(\mf{g})\rt \mf{g}\oplus E$ over $\mf{g}$.
\end{proof}
\begin{remark}\label{fm:rem:cofibersequenceoncotangent}
The cotangent complex \eqref{eq:cotangentbycobar} comes with a $\mc{U}(\mf{g})$-linear map
$$\xymatrix@C=1.8pc{
L_\mf{g} = \mc{U}(\mf{g})\otimes_A \big(\mm{Sym}_A^{\sgeq 1}\mf{g}[1]\big)[-1]\ar[r] & \mc{U}(\mf{g})
}$$
sending $u\otimes X_1\dots X_n$ to zero when $n>1$ and to $u\cdot X_1$ when $n=1$. The \emph{Koszul complex} $K(\mf{g})$ of $\mf{g}$ is the mapping cone of this map. It fits onto a cofiber sequence
$$\xymatrix{
L_\mf{g}\ar[r] & \mc{U}(\mf{g})\ar[r] & K(\mf{g}).
}$$
Unraveling the definitions, $K(\mf{g})$ can be identified with $\mc{U}(\mf{g})\otimes_A \mm{Sym}_A\mf{g}[1]$, with differential given by formula \eqref{fm:eq:koszuldifferential}, but where the term in the second line is also included when $n=1$. The Chevalley-Eilenberg complex $C^*(\mf{g}, E)$ can be identified with $\Hom\hspace{1pt}_{\mc{U}(\mf{g})}(K(\mf{g}), E)$, so that the above cofiber sequence induces (as shift of) the fiber sequence \eqref{eq:ceismapsfromcot} on mapping complexes.
\end{remark}
\begin{remark}\label{rem:ceasext}
The composite map
\begin{equation}\label{fm:diag:cofiberoncotangent}\vcenter{\xymatrix@C=3pc{
L_\mf{g}=\mc{U}(\mf{g})\otimes_A \big(\mm{Sym}_A^{\sgeq 1}\mf{g}[1]\big)[-1]\ar[r] & \mc{U}(\mf{g})\ar[r]^{u\mapsto u\cdot 1} & A
}}\end{equation}
is equal to zero, so that there is a $\mc{U}(\mf{g})$-linear map $K(\mf{g})\rt A$. When $\mf{g}$ is $A$-cofibrant, this map is a weak equivalence. Indeed, the PBW filtration on $\mc{U}(\mf{g})$ (see the proof of Lemma \ref{lem:pbw}) and the filtration on $\mm{Sym}_A\mf{g}[1]$ by polynomial degree determine a total filtration on the Koszul complex $\mc{K}(\mf{g})$. The map on the associated graded is the obvious projection
$$\xymatrix@C=1.4pc{
\mm{Sym}_A(\mf{g}[0, 1])=\mm{Sym}_A\mf{g}\otimes_A \mm{Sym}_A \mf{g}[1] \ar[r] & A
}$$
from the symmetric algebra on the cone $\mf{g}[0, 1]$, which is a weak equivalence.

In other words, $C^*(\mf{g}, E)$ is a model for the derived mapping space $\Hom_{\mc{U}(\mf{g})}(A, E)$. The lax symmetric monoidal structure on $C^*(\mf{g}, -)$ arises from the fact that $A$ is a cocommutative coalgebra in $\dgRep_\mf{g}$.
\end{remark}

\subsection{The free case}\label{fm:sec:ceoffree}
When $\mf{g}=F(V)$ is the free dg-Lie algebroid generated by a cofibrant dg-$A$-module over $T_A$, we now have two different (but weakly equivalent) descriptions of the cotangent complex $L_\mf{g}$: Example \ref{fm:ex:cotangentcomplexoffreealgebroid} simply evaluates the left Quillen functor \eqref{diag:kaehler} on $\mf{g}$ itself, while Proposition \ref{prop:cotangentofliealgebroid} computes the value of \eqref{diag:kaehler} on the `cobar' resolution $Q(\mf{g})$. Of course, the first description is significantly smaller than the second.

When $\mf{g}=F(V)$ is a free Lie algebroid, there is a canonical section of the map $Q(\mf{g})\rto{\sim} \mf{g}$, induced by the canonical inclusion
$$\xymatrix{
V\ar@{^{(}->}[r] & \mf{g}\ar[r] & \Big(\mm{Sym}_A\mf{g}[1]\Big)[-1]\subseteq Q(\mf{g}).
}$$
Applying \eqref{diag:kaehler} to this section produces a weak equivalence between the two models for the cotangent complex $L_\mf{g}$, given by the $\mc{U}(\mf{g})$-linear extension of the above inclusion
\begin{equation}\label{diag:comparisonofcotangent}\smash{\xymatrix{
\mc{U}(\mf{g})\otimes_A V\ar[r]^-{\sim} & \mc{U}(\mf{g})\otimes_A \Big(\mm{Sym}^{\sgeq 1}_A \mf{g}[1]\Big)[-1].
}}\end{equation}
Dually, restriction along this map induces a weak equivalence map to a significantly smaller complex
$$\xymatrix{
\kappa\colon C^*(\mf{g}, A)\ar[r]^-\simeq & A^V = A\oplus_{\rho^\vee} \Hom_A\big(V[1], A\big); \hspace{4pt} \alpha\ar@{|->}[r] & \Big(\alpha(1), \alpha\big|_{V[1]}\Big).
}$$
The codomain $A^V$ has the natural structure of a commutative dg-algebra, the \emph{square zero extension} of $A$ by $V[1]^\vee=\Hom_A(V[1], A)$ classified by the map
$$\xymatrix{
\rho^\vee\colon \Omega_A\ar[r] & V^\vee; \hspace{4pt} d_{\dR}a\ar@{|->}[r] & \Big(v\mapsto \rho(v)(a)\Big).
}$$
In other words, $A^V$ fits into a (homotopy) pullback square of cdgas
\begin{equation}\label{diag:squarezero}\vcenter{\xymatrix{
A^V= A\oplus_{\rho^\vee} V[1]^\vee\ar[r] \ar[d] & A\oplus V[0, 1]^\vee\ar[d]\\
A\ar[r]_{\rho^\vee} & A\oplus V^\vee.
}}\end{equation}
Using this pullback square, one can easily check that the functor
$$\xymatrix{
A^{(-)}\colon \dgMod_A/T_A\ar[r] & \big(\dgCAlg_k/A\big)^\op;  V\ar@{|->}[r] & A^V= A\oplus_{\rho^\vee} V[1]^\vee.
}$$
is a left Quillen functor from the tame model structure on dg-$A$-modules to the usual (projective) model structure on cdgas over $A$. Its right adjoint sends $B\rt A$ to the mapping fiber of the map 
$$\xymatrix{
T_A = \mm{Der}_k(A, A)\ar[r] & \mm{Der}_k(B, A),
}$$
together with its natural projection to $T_A$. We may therefore summarize the previous discussion by the following result:
\begin{corollary}\label{cor:ceoffree}
There is a natural transformation to a right Quillen functor
$$\xymatrix@C=4pc{
\dgMod_A/T_A \ar@/^1.5pc/[r]^{C^*\circ F}_{}="s"\ar@/_1.5pc/[r]_{A^{(-)}}^{}="t" \ar@{=>}"s";"t"^\kappa & \big(\dgCAlg_k/A\big)^\op
}$$
which is a weak equivalence when restricted to cofibrant dg-$A$-modules over $T_A$.
\end{corollary}
Let us finally turn to the proof of Proposition \ref{prop:deformationadjunction}:
\begin{proof}[Proof (of Proposition \ref{prop:deformationadjunction})]
Consider the functor 
$$\xymatrix{
\dgLieAlgd_A\ar[r]^-{C^*} & \big(\dgCAlg_k/A\big)^\op\ar[r]^-\ker & \big(\dgMod_k\big)^\op
}$$
sending a dg-Lie algebroid to the kernel of the map $C^*(\mf{g})\rt A$. By Proposition \ref{prop:cotangentofliealgebroid}, this functor can be identified with the composite
$$\xymatrix{
\dgLieAlgd_A\ar[r]^-{Q} & \dgLieAlgd_A\ar[r]^-{\Upsilon} & \dgRep_{T_A}\ar[rr]^-{\Hom_{\mc{U}(T_A)}(-, A)} & & \big(\dgMod_k\big)^\op.
}$$
The functor $Q$ sends an $A$-cofibrant dg-Lie algebroid to its cofibrant replacement, the functor $\Upsilon$ is the functor \eqref{diag:kaehler} sending $\mf{g}$ to $\mc{U}(T_A)\otimes_{\mc{U}(\mf{g})} \Upsilon_\mf{g}$ and the last functor takes the maps from a $T_A$-representation to $A$. Since the last two functors are left Quillen functors, it follows that $C^*$ preserves weak equivalences between $A$-cofibrant dg-Lie algebroids. We therefore obtain functors of $\infty$-categories
$$\xymatrix{
\LieAlgd_A^!\ar[r]^-{C^*} & \big(\CAlg_k/A\big)^\op\ar[r]^-\ker & \big(\Mod_k\big)^\op.
}$$
The composition of these two functor preserves all colimits. Since the functor $\ker\colon \CAlg_k/A\rt \Mod_k$ detects equivalences and preserves all limits, it follows that $C^*$ preserves all colimits. By the adjoint functor theorem, it follows that $C^*$ admits a right adjoint $\mf{D}\colon \big(\CAlg_k/A\big)^\op\rt \LieAlgd_A^!$.

It remains to describe this right adjoint $\mf{D}$, at least at the level of the underlying tame $A$-modules. To this end, observe that the composite
\begin{equation}\label{diag:forgetafterkoszul}\xymatrix{
\big(\CAlg_k/A\big)^\op\ar[r]^-{\mf{D}} & \LieAlgd_A^!\ar[r]^{U} & \Mod_A^!/T_A
}\end{equation}
is right adjoint to the functor $C^*\circ F\colon \Mod_A^!/T_A\rt \big(\CAlg_k/A\big)^\op$. Corollary \ref{cor:ceoffree} provides a natural equivalence $\kappa\colon C^*\circ F\rt A^{(-)}$, so that $U\circ \mf{D}$ is equivalent to the derived right adjoint to the left Quillen functor $A^{(-)}$. Because $A$ is cofibrant over $k$, the discussion preceding Corollary \ref{cor:ceoffree} shows that this derived functor sends $B\rt A$ to the $A$-linear dual of $L_A\rt L_{A/B}$.
\end{proof}
\begin{remark}\label{fm:rem:dualityfunctor}
The functor $\mf{D}$ does not admit a straightforward point-set description. However, when $B\rt A$ is a \emph{cofibration} of commutative dg-algebras, its image under $\mf{D}$ can be identified with the sub-dg-Lie algebroid $T_{A/B}:=\mm{Der}_B(A, A)\rt T_A$ of $B$-linear derivations of $A$. To see this, note that there is a natural diagram of commutative dg-$k$-algebras over $A$
$$\xymatrix{
B\ar[r]^-{f} \ar[rd]_-\phi & C^*(T_{A/B})\ar[r]\ar[d] & C^*(\tilde{T}_{A/B})\ar[ld]\\
& A &
}$$
where $\tilde{T}_{A/B}\rt T_{A/B}$ is a cofibrant replacement. For any $b\in B$, its image under $f$ is the map
$$\xymatrix@C=4pc{
\mm{Sym}_A\big(T_{A/B}[1]\big)\ar[r]^-{T_{A/B}\mapsto 0} & A\ar[r]^{\phi(b)\cdot (-)} & A
}$$
The map $f$ respects the differential because the derivations in $T_{A/B}$ are $B$-linear. The composition $B\rt C^*(\tilde{T}_{A/B})$ is adjoint to a map $\tilde{T}_{A/B}\rt \mf{D}(B\rt A)$. At the level of the underlying (tame) $A$-modules, this map is simply given by the composite map
$$\xymatrix{
\tilde{T}_{A/B}\ar[r]^\sim & T_{A/B}\ar[r] & L_{A/B}^\vee
}$$
Because $B\rt A$ was a cofibration, the second map is a weak equivalence. It follows that the dg-Lie algebroid $\tilde{T}_{A/B}$ (and therefore $T_{A/B}$) is weakly equivalent to $\mf{D}(B\rt A)$.
\end{remark}

\section{Koszul duality}\label{sec:koszulduality}
The purpose of this section is to show that the adjunction from Proposition \ref{prop:deformationadjunction}
$$\xymatrix{
C^*\colon \LieAlgd^!_A\ar@<1ex>[r] & \big(\CAlg_k/A\big)^\op\colon \mf{D}\ar@<1ex>[l]
}$$
is not too far from being an equivalence: when restricted to suitably \emph{free} tame dg-Lie algebroids, the functor $C^*$ is fully faithful.
\begin{proposition}\label{fm:prop:koszulduality}
Let $A$ be a cofibrant connective commutative dg-$k$-algebra and let $\mf{g}$ be a cofibrant dg-Lie algebroid in the tame model structure. Suppose $\mf{g}$ satisfies the following conditions:
\begin{itemize}
\item[(i)] As a graded $A$-module, $\mf{g}$ is free on a set of generators $x_i$.
\item[(ii)] There are finitely many $x_i$ in each single degree, and no generators of (homological) degree $\geq 0$.
\end{itemize}
Then the unit map $\mf{g}\rt \mf{D}C^*(\mf{g})$ can be identified with the map of dg-$A$-modules
$$\xymatrix{
\mf{g}\ar[r] & \mf{g}^{\vee\vee}
}$$
from $\mf{g}$ into its $A$-linear bidual.
\end{proposition}
\begin{corollary}\label{cor:koszuldualitycoherent}
Let $A$ be a cofibrant commutative dg-$k$-algebra which is bounded. Then $C^*\colon \LieAlgd_A^!\rt \big(\CAlg_k/A\big)^\op$ is fully faithful on all tame dg-Lie algebroids satisfying conditions (i) and (ii) in Proposition \ref{fm:prop:koszulduality}.
\end{corollary}
\begin{proof}
Forgetting about the differential, we can write $\mf{g} = \bigoplus_{n<0} A^{\oplus k_n}[n]$. The unit map $\mf{g}\rt \mf{g}^{\vee\vee}$ can then be identified with the map
$$\xymatrix{
\bigoplus_{n<0} A^{\oplus k_n}[n]\ar[r] & \prod_{n<0} A^{\oplus k_n}[n].
}$$
When $A$ is bounded, this map is an isomorphism.
\end{proof}
\begin{corollary}\label{cor:koszuldualityperfect}
Let $A$ be a cofibrant commutative dg-$k$-algebra which is eventually coconnective. Then $C^*\colon \LieAlgd_A\rt \big(\CAlg_k/A\big)^\op$ is fully faithful on all Lie algebroids that can be modeled by \emph{projectively} cofibrant dg-Lie algebroids satisfying conditions (i) and (ii) of Proposition \ref{fm:prop:koszulduality}.
\end{corollary}
\begin{proof}
When $\dgLieAlgd_A$ is equipped with the projective model structure, it suffices to verify that the map $\mf{g}\rt \mf{g}^{\vee\vee}$ is a quasi-isomorphism for any projectively cofibrant dg-Lie algebroid satisfying conditions (i) and (ii). Since $\mf{g}$ is a projectively cofibrant dg-$A$-module,  $\mf{g}^{\vee\vee}$ is a model for the derived bidual of $\mf{g}$. It therefore suffices to pick a weak equivalence $A\rt A'$ to a bounded cdga and verify that the map $\mf{g}\otimes_A A'\rt (\mf{g}\otimes_A A')^{\vee\vee}$ is a quasi-isomorphism. But because $A'$ is bounded, this map is an isomorphism, as in the previous proof.
\end{proof}
To prove Proposition \ref{fm:prop:koszulduality}, let us start by considering the map of commutative dg-algebras over $A$
\begin{equation}\label{diag:classifyingmapforunit}\xymatrix{
c\colon C^*(\mf{g})\ar[r] & A^{\mf{g}} = A\oplus_{\rho^\vee} \mf{g}[1]^\vee
}\end{equation}
which sends $\alpha\colon \mm{Sym}_A\mf{g}[1]\rt A$ to $\big(\alpha(1), \alpha\big|_{\mf{g}[1]}\big)$. In the proof of Proposition \ref{prop:deformationadjunction}, we have seen that the functor $\mf{g}\mapsto A^\mf{g}$ was a right Quillen functor, whose derived left adjoint sent
$$\xymatrix{
\big(B\rt A\big)\ar@{|->}[r] & \big(L_{A/B}^\vee\rt L_A^\vee=T_A\big).
}$$
\begin{lemma}\label{lem:identificationunderadjunction}
Let $\mf{g}$ be a tamely cofibrant dg-Lie algebroid over $A$ and consider the $A$-linear map $\mf{g}\rt L_{A/C^*(\mf{g})}^\vee$ adjoint to \eqref{diag:classifyingmapforunit}. This map is equivalent to the $A$-linear map underlying the unit map $\mf{g}\rt \mf{D}C^*(\mf{g})$.
\end{lemma}
\begin{proof}
Let $F\colon \Mod_A^!/T_A\leftrightarrows \LieAlgd_A^!\colon U$ be the free-forgetful adjunction. The $A$-linear map $U(\mf{g})\rt U\mf{D}C^*(\mf{g})$ in $\Mod_A^!/T_A$ corresponds by adjunction to the map
$$\xymatrix{
C^*(\mf{g})\ar[r] & C^*(F(\mf{g}))
}$$
in $\CAlg/A$. The composition of this map with the equivalence $\kappa\colon C^*(F(\mf{g}))\rt A^\mf{g}$ from Corollary \ref{cor:ceoffree} is exactly the map \eqref{diag:classifyingmapforunit}. This means that the maps 
$$
\mf{g}\rt L_{A/C^*(\mf{g})}^\vee \quad \text{and} \quad U(\mf{g})\rt U\mf{D}C^*(\mf{g})
$$
are identified under the adjoint equivalence between $U\mf{D}$ and the functor sending $B\rt A$ to $L_{A/B}^\vee$.
\end{proof}
To use Lemma \ref{lem:identificationunderadjunction}, we will have to compute the relative cotangent complex of the map $C^*(\mf{g})\rt A$. Unfortunately, $C^*(\mf{g})$ has the structure of a power series algebra, which means that $C^*(\mf{g})$ is not cofibrant and computing its cotangent complex requires some effort. Let us therefore introduce the following `global' variant of the Chevalley-Eilenberg complex:
\begin{construction}\label{fm:cons:globalizedCE}
Let $A$ be a cofibrant commutative dg-$k$-algebra and let $\mf{g}$ be a (tamely) cofibrant dg-Lie algebroid over $A$ satisfying conditions (i) and (ii) of Proposition \ref{fm:prop:koszulduality}. Let
$$
C^*_\mm{poly}(\mf{g}) := \mm{Sym}_A \big(\mf{g}[1]^\vee\big) \subseteq C^*(\mf{g})
$$
be the graded-subalgebra of $C^*(\mf{g})$ consisting of graded $A$-linear maps $\mm{Sym}_A\mf{g}[1] \rt A$ that vanish on some $\mm{Sym}^{\sgeq n}_A \mf{g}[1]$. This graded subalgebra of $C^*(\mf{g})$ is closed under the Chevalley-Eilenberg differential of $C^*(\mf{g})$, which sends a function vanishing on $\mm{Sym}^{\sgeq n}_A \mf{g}[1]$ to a function vanishing on $\mm{Sym}^{\sgeq n+1}_A \mf{g}[1]$.
\end{construction}
\begin{example}\label{fm:ex:polynomialintopowerseries}
Let $\mf{g}=A^{\oplus n}[-1]$ be the trivial dg-Lie algebroid on $n$ generators of degree $-1$. Then $C^*(\mf{g})$ is isomorphic to the ring of power-series $A[[x_1, ..., x_n]]$ and the inclusion $C^*_\mm{poly}(\mf{g})\subseteq C^*(\mf{g})$ is the inclusion of the polynomial algebra $A[x_1, ..., x_n]\subseteq A[[x_1, ..., x_n]]$.
\end{example}
\begin{warning}
The commutative dg-algebra $C^*_\mm{poly}(\mf{g})$ is not a homotopy invariant.
\end{warning}
\begin{lemma}\label{fm:lem:cotangentforglobalce}
Let $A$ be a cofibrant commutative dg-$k$-algebra and let $\mf{g}$ be a dg-Lie algebroid over $A$ such that $\mf{g}\cong \bigoplus_{n<0} A^{\oplus k_n}[n]$ as a graded $A$-module. Then the following hold:
\begin{enumerate}[leftmargin=*]
 \item $C^*_\mm{poly}(\mf{g})$ is a cofibrant commutative dg-$k$-algebra.
 \item the map on K\"ahler differentials (relative to the base cdga $k$)
 $$\xymatrix{
 \Omega_{C^*_\mm{poly}(\mf{g})}\otimes_{C^*_\mm{poly}(\mf{g})} A\ar[r] &  \Omega_A
 }$$
 can be identified with the projection map $\Omega_A\oplus \mf{g}[1]^\vee\rt \Omega_A$. Here $\Omega_A\oplus \mf{g}[1]^\vee$ is the mapping fiber of the dual of the anchor map, i.e.\ it has differential
 $$\xymatrix{
 \dau\big(d_\mm{dR}(a), \alpha\big) = \Big(d_\mm{dR}(\dau_Aa), \dau_{\mf{g}[1]^\vee}(\alpha) + \rho^\vee(d_\mm{dR}a)\Big)
 }$$
 where $\rho^\vee\colon \Omega_A \rt \mf{g}^\vee$ is the adjoint of the anchor map $\mf{g}\rt T_A$.
\end{enumerate}
\end{lemma}
\begin{proof}
Since $\mf{g}$ is given by $\bigoplus_{i<0} A^{\oplus k_i}[i]$ as a graded $A$-module, $C^*_\mm{poly}(\mf{g})$ is a polynomial algebra over $A$, generated by the free module $\mf{g}[1]^\vee$. This module has generators in degrees $\geq 0$ and $A$ is cofibrant, so that $C^*_\mm{poly}(\mf{g})$ is the retract of a connective graded polynomial ring over $k$. This implies that $C^*_\mm{poly}(\mf{g})$ is cofibrant.

For the second assertion, note that $C^*_\mm{poly}(\mf{g})$ is freely generated over $A$ by the graded $A$-module $\mf{g}[1]^\vee$. It follows that 
$$
\Omega_{C^*_\mm{poly}(\mf{g})}\otimes_{C^*_\mm{poly}(\mf{g})} A\cong \Omega_A\oplus \mf{g}[1]^\vee
$$
as a graded $A$-module. The map $C^*_\mm{poly}(\mf{g})\rt A$ sends all $\mf{g}[1]^\vee$ to zero, so that the induced map on K\"ahler differentials is just the projection $\Omega_A\oplus \mf{g}[1]^\vee\rt \Omega_A$. Furthermore, the Chevalley-Eilenberg differential is given by
\begin{align*}
A\ni a &\; \mapsto \dau_A a + \rho^\vee(d_\mm{dR}a)\\
\mf{g}[1]^\vee \ni \alpha &\; \mapsto \dau_{\mf{g}[1]^\vee}\alpha \quad \mod \big(\mf{g}[1]^\vee\big)^2.
\end{align*}
This shows that the differential on $ \Omega_A\oplus \mf{g}[1]^\vee$ is given as in the lemma.
\end{proof}
\begin{lemma}\label{fm:lem:addingvariablestopowerseries}
Let $A$ be a nonnegatively graded commutative $\field{Q}$-algebra, let $V$ be a finite dimensional $\field{Q}$-vector space and let $W$ be a degreewise finite-dimensional graded $\field{Q}$-vector space, concentrated in degrees $<0$. Then there is a natural isomorphism of graded-commutative $A$-algebras
$$\xymatrix{
\Hom_\field{Q}\big(\mm{Sym}_\field{Q}V, A\big)\otimes_\field{Q} \mm{Sym}_\field{Q}(W^\vee)\ar[r] & \Hom_\field{Q}\big(\mm{Sym}_\field{Q}(V\oplus W), A\big)
}$$ 
where $\mm{Sym}_\field{Q} W^\vee$ is the graded polynomial algebra on the dual vector space of $W$.
\end{lemma}
\begin{proof}
Observe that there is an isomorphism of graded cocommutative coalgebras $\mm{Sym}_\field{Q}(V\oplus W)\cong \mm{Sym}_\field{Q}V\otimes_{\field{Q}} \mm{Sym}_\field{Q}W$. There is a natural map of graded-commutative algebras
$$\xymatrix{
\Hom_\field{Q}\big(\mm{Sym}_\field{Q}V, A\big)\otimes_\field{Q} \Hom_\field{Q}(\mm{Sym}_\field{Q} W, \field{Q})\ar[r]^-{\mu} & \Hom_\field{Q}\big(\mm{Sym}_\field{Q}V\otimes_{\field{Q}} \mm{Sym}_\field{Q} W, A\big).
}$$
sending two maps $\alpha\colon \mm{Sym}_\field{Q}V\rt A$ and $\beta\colon \mm{Sym}_\field{Q} W\rt \field{Q}$ to $\alpha\otimes \beta$. Using that $\mm{Sym}_\field{Q} W$ is free on generators of degrees $<0$, with finitely many generators in each degree, one can identify $\mm{Sym}_\field{Q}(W^\vee)\simeq \Hom_\field{Q}(\mm{Sym}_\field{Q} W, \field{Q})$. Using this, it follows that $\mu$ is an isomorphism.
\end{proof}
\begin{lemma}\label{fm:lem:polynomialintopower}
Let $A$ be a cofibrant connective commutative dg-$k$-algebra and let $\mf{g}$ be as in Proposition \ref{fm:prop:koszulduality}. Then the map $C^*_\mm{poly}(\mf{g})\rt C^*(\mf{g})$ induces an equivalence on cotangent complexes over $k$
$$\xymatrix{
L_{C^*_\mm{poly}(\mf{g})}\otimes_{C^*_\mm{poly}(\mf{g})} A\ar[r]^-\simeq & L_{C^*(\mf{g})}\otimes_{C^*(\mf{g})} A.
}$$
\end{lemma}
\begin{proof}
Consider the trivial cofibration, followed by a fibration
$$\xymatrix{
0\ar[r]^-{\sim} & \mf{h}=F(\mf{g}[0,-1])\ar@{->>}[r] & \mf{g}
}$$
where $\mf{h}$ is the free dg-Lie algebroid on the map $\mf{g}[0, -1]\rt \mf{g}\rt T_A$ from the path space of $\mf{g}$. Let $V$ be the free graded $\field{Q}$-vector space spanned by the generators $x_i$ of $\mf{g}$, so that $\mf{g}=A\otimes_\field{Q} V$. As a graded Lie algebroid, $\mf{h}$ is then freely generated by the graded $\field{Q}$-vector space $V[0, -1]$. Consequently, the map $\mf{h}\rt \mf{g}$ is given without differentials by the $A$-linear extenion of a map from the free Lie algebra
$$\xymatrix{
\mf{h}=A\otimes\mm{Lie}\big(V[0, -1]\big)\ar[r] & A\otimes V= \mf{g}
}$$
which sends $V$ to the generators of $\mf{g}$ and $V[-1]$ to zero. This map has a splitting, induced by the inclusion $V\rt \mm{Lie}(V[0, -1])$, so that $\mf{h}\rt \mf{g}$ can be identified with
$$\xymatrix{
\mf{h} = A\otimes (V\oplus W)\ar[r]^-{(\mm{id}, 0)} & A\otimes V = \mf{g}.
}$$
Here $W$ is a graded $\field{Q}$-vector space isomorphic to $\mm{Lie}(V[0, -1])/V$, which is degreewise finite dimensional and concentrated in degrees $<-1$. Indeed, $V$ is degreewise finite dimensional and concentrated in degrees $<0$, so that $\mm{Lie}(V[0, -1])$ has these properties as well and $V\rt \mm{Lie}(V[0, -1])$ is an isomorphism in degree $-1$.

Let us now consider the commutative diagram of cdgas associated to $\mf{h}\rt \mf{g}$
\begin{equation}\label{fm:diag:thecrucialpushout}\vcenter{\xymatrix{
C^*_\mm{poly}(\mf{g})\ar[d]\ar[r] & C^*_\mm{poly}(\mf{h})\ar[d]\ar[r] & A\ar[d]\\
C^*(\mf{g})\ar[r] & C^*(\mf{h})\ar[r]_-\sim & A.
}}\end{equation}
The right bottom map is a weak equivalence since $\mf{h}$ is cofibrant and weakly contractible. The map $C^*_\mm{poly}(\mf{g})\rt C^*_\mm{poly}(\mf{h})$ can be identified with a map of polynomial algebras
$$\xymatrix{
A\otimes_\field{Q} \mm{Sym}_\field{Q}(V[1]^\vee)\ar[r] & A\otimes_\field{Q} \mm{Sym}_\field{Q}\big((V\oplus W)[1]^\vee).
}$$
It follows that $C^*_\mm{poly}(\mf{h})$ is freely generated over $C^*_\mm{poly}(\mf{g})$ by $W[1]^\vee$, which is degreewise finite dimensional and concentrated in degrees $\geq 1$. In particular, $C^*_\mm{poly}(\mf{g})\rt C^*_\mm{poly}(\mf{h})$ is a cofibration of cdgas.

On the other hand, the map $C^*(\mf{g})\rt C^*(\mf{h})$ is given without differentials by the natural map
$$\smash{\xymatrix{
\Hom_\field{Q}\Big(\mm{Sym}_\field{Q} V[1], A\Big)\ar[r] & \Hom_\field{Q}\Big(\mm{Sym}_\field{Q} (V[1]\oplus W[1]), A\Big).
}}$$
It now follows from Lemma \ref{fm:lem:addingvariablestopowerseries} that the left square in \eqref{fm:diag:thecrucialpushout} is a (homotopy) pushout square of cdgas. Its image under $L_{(-)}\otimes_{(-)} A$
\begin{equation}\label{fm:diag:pushoutofcotcomplexes}\vcenter{\xymatrix{
L_{C^*_\mm{poly}(\mf{g})}\otimes_{C^*_\mm{poly}(\mf{g})} A\ar[d]\ar[r] & L_{C^*_\mm{poly}(\mf{h})}\otimes_{C^*_\mm{poly}(\mf{h})} A\ar[d]\\
L_{C^*(\mf{g})}\otimes_{C^*(\mf{g})} A\ar[r] & L_{C^*(\mf{h})}\otimes_{C^*(\mf{h})} A
}}\end{equation}
is a homotopy pushout square as well. Since $C^*(\mf{h})\rt A$ is a weak equivalence, the map $L_{C^*(\mf{h})}\rt L_A$ is a weak equivalence. 
On the other hand, the map $L_{C^*_\mm{poly}(\mf{h})}\otimes_{C^*_\mm{poly}(\mf{h})} A\rt L_A$ is identified with the projection map
$$\xymatrix{
\Omega_A\oplus \mf{h}[1]^\vee\ar[r] & \Omega_A
}$$
by Lemma \ref{fm:lem:cotangentforglobalce}. The kernel of this map is contractible, since $\mf{h}$ is a cofibrant contractible dg-$A$-module. It follows that the right vertical map in Diagram \eqref{fm:diag:pushoutofcotcomplexes} is an equivalence, so that the left map is an equivalence as well.
\end{proof}
\begin{proof}[Proof (of Proposition \ref{fm:prop:koszulduality})]
By Lemma \ref{lem:identificationunderadjunction}, it suffices to show that the map $\mf{g}\rt L_{A/C^*}^\vee$ is adjoint to a weak equivalence $L_{A/C^*(\mf{g})}\rt \mf{g}^\vee$. This map fits into a sequence of maps
$$\xymatrix{
L_{A/C^*_\mm{poly}(\mf{g})} \ar[r] & L_{A/C^*(\mf{g})}  \ar[r] &  \mf{g}^\vee
}$$
classifying the composite map of commutative dg-algebras over $A$
$$\xymatrix{
C^*_\mm{poly}(\mf{g}) \ar[r] & C^*(\mf{g})\ar[r]^c & A^\mf{g}
}$$
where $c$ is as in \eqref{diag:classifyingmapforunit}. The map $L_{A/C^*_\mm{poly}(\mf{g})}\rt L_{A/C^*(\mf{g})}$ is an equivalence by Lemma \ref{fm:lem:polynomialintopower}, so it suffices to show that $L_{A/C^*_\mm{poly}(\mf{g})} \rt  \mf{g}^\vee$ is an equivalence. This map can be computed explicitly: the map 
$$\xymatrix{
C_\mm{poly}^*(\mf{g})\ar[r] & A^\mf{g}=A\oplus_{\rho^\vee} \mf{g}[1]^\vee
}$$
is simply the quotient of $C^*_\mm{poly}(\mf{g})$ by the augmentation ideal $(\mf{g}[1]^\vee)^2$. Unwinding the definitions, e.g.\ using the pullback square \eqref{diag:squarezero}, one finds the following description of the classifying map $L_{A/C^*_\mm{poly}(\mf{g})} \rt \mf{g}^\vee$: it is  the canonical map from the mapping cone of
$$\xymatrix{
\Omega_{C^*_\mm{poly}(\mf{g})}\otimes_{C^*_\mm{poly}(\mf{g})} A \cong \Omega_A\oplus \mf{g}[1]^\vee\ar[r] & \Omega_A
}$$
to $\mf{g}^\vee$. This map is a weak equivalence, which concludes the proof.
\end{proof}

\section{Formal moduli problems}\label{sec:formmod}
We will now use the results of Section \ref{sec:koszulduality} to establish an equivalence between Lie algebroids and formal moduli problems (Definition \ref{def:formalmoduli}):
\begin{theorem}\label{thm:formalmoduliperfect}
Let $A\in \CAlg^{\dg, \sgeq 0}_k$ be cofibrant. Then there is an adjoint pair of functors
$$\xymatrix{
\mm{MC}\colon \LieAlgd_A\ar@<1ex>[r] & \cat{FormMod}_A\ar@<1ex>[l] \colon T_{A/}
}$$
which is an equivalence whenever $A$ is eventually coconnective.
\end{theorem}
\begin{theorem}\label{thm:formalmodulicoherent}
Suppose that $A\in \CAlg^{\dg, \sgeq 0}_k$ is cofibrant, bounded and coherent. Then there is an equivalence of $\infty$-categories
$$\xymatrix{
\mm{MC}\colon \LieAlgd^!_A\ar@<1ex>[r] & \cat{FormMod}^!_A\ar@<1ex>[l] \colon T_{A/}
}$$
between the $\infty$-category of pro-coherent Lie algebroids over $A$ and the $\infty$-category of pro-coherent formal moduli problems $\CAlg^{\mm{sm, coh}}/A\rt \sS$ (Section \ref{sec:intro}).
\end{theorem}
These theorems follow formally from Corollary \ref{cor:koszuldualityperfect} and \ref{cor:koszuldualitycoherent}, by means of a general procedure due to Lurie \cite{lur11X} that we will briefly recall.

\subsection{Generators}
Categories of chain complexes or spectra endowed with a certain algebraic structure often admit a presentation in terms of generators and relations.
\begin{definition}\label{def:goodobjects}
Let $\Xi$ be a locally presentable $\infty$-category equipped with a collection of right adjoint functors
$$\smash{\xymatrix{
e_\alpha\colon \Xi\ar[r] & \cat{Sp}
}}$$
to the $\infty$-category of spectra. The left adjoint to $e_\alpha$ sends the map $\field{S}^n\rt 0$ in $\cat{Sp}$ to a map in $\Xi$ that we will denote by $K_{\alpha, n}\rt \emptyset$. We will say that an object $X\in \Xi$ is \emph{good} if it admits a finite filtration
$$\smash{\xymatrix{
\emptyset=X^{(0)}\ar[r] & X^{(1)}\ar[r] & \dots\ar[r] & X^{(n)}
}}$$
where for each $i$, there is an $\alpha$ and a $n\leq -2$, together with a pushout square
\begin{equation}\label{diag:addingcell}\vcenter{\xymatrix{
K_{\alpha, n}\ar[r]\ar[d] & X^{(i-1)}\ar[d]\\
\emptyset\ar[r] & X^{(i)}.
}}\end{equation}
Let $\Xi^\mm{good}\subseteq \Xi$ be the full subcategory on the good objects; it is the smallest subcategory of $\Xi$ which contains $\emptyset$ and is closed under pushouts along the maps $K_{\alpha, n}\rt \emptyset$ with $n\leq -2$.
\end{definition}
\begin{proposition}[{\cite[Theorem 1.3.12]{lur11X}}]\label{prop:luriedefthy}
Let $(\Xi, e_\alpha)$ be as in Definition \ref{def:goodobjects}. Suppose that each $e_\alpha$ preserves small sifted homotopy colimits and that a map $f$ in $\Xi$ is an equivalence if and only if each $e_\alpha(f)$ is an equivalence of spectra. Then the right adjoint functor
$$\xymatrix{
\Psi\colon \Xi\ar[r] & \mm{PSh}(\Xi^\mm{good}); \hspace{4pt} X\ar@{|->}[r] & \Map_{\Xi}(-, X)
}$$
is fully faithful, with essential image consisting of those (space-valued) presheaves $F$ satisfying the following two conditions:
\begin{itemize}
\item[(a)] $F(\emptyset)$ is contractible.
\item[(b)] For any $\alpha$ and $n\leq -2$, $F$ sends a pushout square of the form \eqref{diag:addingcell} to a pullback square of spaces.
\end{itemize}
\end{proposition}
This is exactly \cite[Theorem 1.3.12]{lur11X}, replacing the category $\Upsilon^\mm{sm}$ from loc.\ cit.\ by the opposite of $\Xi^\mm{good}$.
\begin{example}\label{ex:obvious1}
Let $\Xi$ be a compactly generated stable $\infty$-category. The collection of functors $\Hom_{\Xi}(K, -)\colon \Xi\rt \cat{Sp}$, for all compact objects $K$, satisfies the condition of Proposition \ref{prop:luriedefthy}. The good objects are exactly the compact objects and a functor $F\colon \Xi^{\omega, \op}\rt \sS$ satisfies conditions (a) and (b) if and only if it is left exact. Proposition \ref{prop:luriedefthy} then reproduces the equivalence
$$
\Xi\simeq \mm{Ind}(\Xi^\omega).
$$
\end{example}
\begin{example}\label{ex:obvious2}
Let $A$ be a commutative dg-algebra (over a field of characteristic zero) and let $\Mod_A$ be the $\infty$-category of $A$-modules. The single functor $e\colon \Mod_A\rt \cat{Sp}$, forgetting the $A$-module structure, satisfies the conditions of Proposition \ref{prop:luriedefthy}. In this case, the good $A$-modules can be presented by the dg-$A$-modules whose underlying graded $A$-module is free on finitely many generators $x_i$ of degree $<0$.

There is an equivalence of $\infty$-categories 
$$\xymatrix{
\Mod_A^\mm{good, op}\ar[r] & \Mod_A^{\mm{f.p.}, \sgeq 0}; \hspace{4pt} E\ar[r] & E[1]^\vee
}$$
to the $\infty$-category of finitely presented connective $A$-modules, i.e.\ dg-$A$-modules generated by finitely many generators of degree $\geq 0$. Combining this equivalence with Proposition \ref{prop:luriedefthy}, one finds that $\Mod_A$ is equivalent to the $\infty$-category of functors
$$\xymatrix{
F\colon \Mod_A^{\mm{f.p.}, \sgeq 0}\ar[r] & \sS
}$$
that send $0$ to a contractible space and preserve pullbacks along the maps $0\rt A[n]$ with $n\geq 1$.
\end{example}

\subsection{Proof of Theorem \ref{thm:formalmoduliperfect}}
Let $A$ be a cofibrant commutative dg-$k$-algebra. By Proposition \ref{prop:monadicity} and Example \ref{ex:obvious2}, the composite forgetful functor
$$\xymatrix{
\LieAlgd_A\ar[r]^U & \Mod_A/T_A\ar[r]^-{\ker} & \Mod_A\ar[r]^-{e} & \cat{Sp}
}$$
preserves small limits and sifted colimits, and detects equivalences. 

The corresponding notion of a \emph{good Lie algebroid} can be described in terms of the (projective) model structure on dg-Lie algebroids as follows: let us say that a dg-Lie algebroid $\mf{g}$ is \emph{very good} if it admits a finite sequence of cofibrations 
$$\xymatrix{
0=\mf{g}^{(0)}\ar[r] & \cdots \ar[r] & \mf{g}^{(n)}=\mf{g},
}$$
each of which is the pushout of a (generating) cofibration with $n_i\leq -2$
\begin{equation}\label{fm:rem:verygoodliealgebroids}\xymatrix{
\mm{Free}\big(\dau\phi\colon A[n_i] \rt T_A\big)\ar[r] & \mm{Free}\big(\phi\colon A[n_i, n_i+1]\rt T_A\big).
}\end{equation}
Here $\phi$ is a map from the cone of $A[n_i]$ to $T_A$, which is determined uniquely by a degree $(n_i+1)$ element of $T_A$. Then the good Lie algebroids can be presented by the very good dg-Lie algebroids over $A$. 
\begin{remark}
Good Lie algebroids may have nontrivial anchor maps, even though the Lie algebroids in \eqref{fm:rem:verygoodliealgebroids} have null-homotopic anchor maps.
\end{remark}
\begin{lemma}\label{fm:lem:verygoodlralgebroids}
Let $\mf{g}$ be a very good dg-Lie algebroid over $A$. Then the following hold:
\begin{enumerate}[leftmargin=*]
\item $\mf{g}$ has a cofibrant underlying dg-$A$-module.
\item Without the differential, $\mf{g}$ is freely generated by a negatively graded finite-dimensional vector space over $T_A$.
\item $\mf{g}$ is isomorphic as a graded $A$-module to $\bigoplus_{n<0} A^{\oplus k_n}[n]$ for some sequence of $k_n\in \field{N}_{\geq 0}$.
\end{enumerate}
\end{lemma}
\begin{proof}
Assertion (1) is obvious and (3) follows immediately from (2). For (2), note that each pushout along a map \eqref{fm:rem:verygoodliealgebroids} freely adds a single generator of degree $<0$ at the level of graded Lie algebroids.
\end{proof}
\begin{proof}[Proof (of Theorem \ref{thm:formalmoduliperfect})]
Let us denote the $\infty$-category of presheaves satisfying conditions (a) and (b) of Proposition \ref{prop:luriedefthy} by
$$
\cat{E}\subseteq \mm{PSh}(\LieAlgd_A^\mm{good}).
$$
We then have an equivalence $\LieAlgd_A\simeq \cat{E}$, so that it suffices to produce the required adjunction (equivalence) between $\cat{E}$ and the $\infty$-category of formal moduli problems. To this end, recall that the good Lie algebroids form the smallest subcategory of $\LieAlgd_A$ that contains $0$ and is closed under pushouts against the maps 
\begin{equation}\label{diag:generatinggoodmaps}\xymatrix{
F(0\colon A[n]\rt T_A)\ar[r] & 0 & n\leq -2.
}\end{equation}
The functor $C^*$ preserves colimits (Proposition \ref{prop:deformationadjunction}) and sends the above maps to the maps $(\mm{id}, 0)\colon A\rt A\oplus A[-1-n]$, for $n\leq -2$. It follows that $C^*$ restricts to a functor
\begin{equation}\label{diag:cerestrictedtogood}\xymatrix{
C^*\colon \LieAlgd_A^\mm{good}\ar[r] & \Big(\CAlg^\mm{sm}_k/A\Big)^\op.
}\end{equation}
The restriction of a formal moduli problem along $C^*$ is a presheaf contained in $\cat{E}$. We therefore obtain a right adjoint functor
$$\xymatrix{
T_{A/}\colon \cat{FormMod}_A\ar[r]^-{(C^*)^*} & \cat{E}\ar[r]^-{\simeq} & \LieAlgd_A.
}$$
If $A$ is eventually coconnective, Corollary \ref{cor:koszuldualityperfect} and Lemma \ref{fm:lem:verygoodlralgebroids} show that \eqref{diag:cerestrictedtogood} is fully faithful. The essential image of $C^*$ contains $A$ and is closed under pullbacks along the maps $(\mm{id}, 0)\colon A\rt A\oplus A[n]$ for $n\geq 1$. Indeed, such pullbacks can dually be computed as pushouts along the images of the maps \eqref{diag:generatinggoodmaps}. The small extension of $A$ form the smallest subcategory of $\CAlg_k/A$ with these two closure properties, so \eqref{diag:cerestrictedtogood} is essentially surjective as well. This implies that the functor $T_{A/}$ is an equivalence.
\end{proof}
\begin{example}\label{ex:tangenttomapofaffines}
Let $f\colon B\rt A$ be a map of connective commutative $k$-algebras and let 
$$\xymatrix{
\spec(B)^\wedge\colon \CAlg_k^\mm{sm}/A \ar[r] & \sS; \hspace{4pt} \tilde{A}\ar@{|->}[r] & \Map(B, \tilde{A})\times_{\Map(B, A)} \{f\}
}$$
be the formal completion of $\spec(B)$ at $A$. This is a formal moduli problem over $A$. Unwinding the definitions, one sees that the associated Lie algebroid is given by $\mf{D}(B)$, which can be identified with the Lie algebroid $T_{A/B}\rt T_A$ of (derived) $B$-linear derivations of $A$ by Remark \ref{fm:rem:dualityfunctor}.
\end{example}
\begin{remark}\label{rem:mcintermsofD}
Suppose that $A$ is eventually coconnective, so that $\CAlg_k^\mm{sm}/A$ is equivalent to the $\infty$-category of good Lie algebroids over $A$. If $\mf{g}$ is a Lie algebroid over $A$, then the formal moduli problem $\mm{MC}_\mf{g}$ is given (up to a natural equivalence) by
$$
\mm{MC}_\mf{g}(B) = \Map_{\LieAlgd_A}\big(\mf{D}(B), \mf{g}\big).
$$
By Remark \ref{ex:tangenttomapofaffines}, one can think of this as the space of flat $\mf{g}$-valued connections on the fiberwise tangent bundle of $A$ over $B$.
\end{remark}

\subsection{Proof of Theorem \ref{thm:formalmodulicoherent}}\label{sec:tamecase}
Theorem \ref{thm:formalmodulicoherent} is proven in exactly the same way: let $\mc{K}^{\sleq 0}\subseteq \Mod_A^!$ be the full subcategory on the compact generators provided by Proposition \ref{prop:tamecompactlygenerated}. For each $K_\alpha\in \mc{K}$, there is a functor
$$\xymatrix{
\LieAlgd_A^!\ar[r]^U & \Mod_A^!/T_A\ar[r]^-\ker & \Mod_A^!\ar[rr]^-{\Hom_A(K_\alpha, -)} & & \cat{Sp}
}$$
which preserves limits and sifted colimits by Proposition \ref{prop:monadicitytame}. All these functor collectively detect equivalences between tame Lie algebroids.

The good tame Lie algebroids then have the following point-set description: they are given by those dg-Lie algebroids obtained from finitely many pushouts along
\begin{equation}\label{fm:rem:verygoodtameliealgebroids}\xymatrix{
\mm{Free}\big(\dau\phi\colon K[-2] \rt T_A\big)\ar[r] & \mm{Free}\big(\phi\colon K[-2, -1]\rt T_A\big)
}\end{equation}
with $K\in\mc{K}$. The assertions of Lemma \ref{fm:lem:verygoodlralgebroids} apply to such dg-Lie algebroids as well, so that Corollary \ref{cor:koszuldualitycoherent} provides fully faithful functor
$$\xymatrix{
C^*\colon \LieAlgd_A^{!, \mm{good}}\ar[r] & \Big(\CAlg_k/A\Big)^\op.
}$$
The images of the maps \eqref{fm:rem:verygoodliealgebroids} under $C^*$ are exactly the maps $A\rt A\oplus E[1]$, where $E$ is a connective coherent $A$-module. Arguing as in the proof of Theorem \ref{thm:formalmoduliperfect}, it follows that $C^*$ induces an equivalence between good tame Lie algebroids over $A$ and the $\infty$-category $\CAlg_k^\mm{sm, coh}/A$ of small extensions of $A$ by coherent $A$-modules.

Restriction along this equivalence provides an equivalence between $\cat{FormMod}_A^!$ and the $\infty$-category of presheaves on $\LieAlgd_A^{!, \mm{good}}$ satisfying conditions (a) and (b) from Proposition \ref{prop:luriedefthy}. In turn, the $\infty$-category of such presheaves is equivalent to $\LieAlgd_A^!$ by Proposition \ref{prop:luriedefthy}.

\section{Representations and quasi-coherent sheaves}\label{sec:representations}
In the previous section, we have seen that there is an equivalence between formal moduli problems over $A$ and Lie algebroids over $A$. Geometrically, one can think of a formal moduli problem as a map of stacks $x\colon \spec(A)\rt X$ that realizes $X$ as an infinitesimal thickening of $\spec(A)$. Any such stack $X$ gives rise to an $\infty$-category of \emph{quasi-coherent sheaves} on $X$ in the usual way (Definition \ref{def:qc}): a quasi-coherent sheaf $F$ is given by a collection of $B$-modules $F_y$ for every $B\in \CAlg_k^\mm{sm}/A$ and every $y\in X(B)$, together with a coherent family of equivalences
$$
F_{\alpha(y)}\simeq B'\otimes_B F_y
$$
for every $\alpha\colon B\rt B'$. In particular, a quasi-coherent sheaf $F$ on $X$ determines an $A$-module $F_x$, by restricting to the canonical point $x\in X(A)$. We will see that $F_x$ carries a representation of $T_{A/X}$. In fact, we will prove the following:
\begin{theorem}\label{fm:thm:quasi-coherentvsrepresentations}
Let $A$ be eventually coconnective and let $X$ be a formal moduli problem over $A$ with associated Lie algebroid $T_{A/X}$. Then there is a fully faithful, symmetric monoidal left adjoint functor
$$\xymatrix{
\Psi_X\colon \QC(X)\ar[r] & \Rep_{T_{A/X}}
}$$
where $\Psi_X(F)$ has underlying $A$-module given by the restriction $F_x$ to the canonical basepoint $x\in X(A)$. Furthermore, the functor $\Psi_X$ induces an equivalence 
$$
\Mod(X)^{\sgeq 0}\simeq \Rep_{T_{A/X}}^{\sgeq 0}
$$
between the connective quasi-coherent sheaves on $X$ and the $T_{A/X}$-representations whose underlying $A$-module is connective.
\end{theorem}
Our proof closely follows the discussion in \cite[Section 2.4]{lur11X}. We will begin by considering representations of good Lie algebroids in Section \ref{fm:sec:representationsofgood}. In fact, for later purposes it will be convenient to work with the tame model structure on $\mf{g}$-representations. Theorem \ref{fm:thm:quasi-coherentvsrepresentations} then follows from a formal argument described in Section \ref{fm:sec:qcoh}.

\subsection{Representations of good Lie algebroids}\label{fm:sec:representationsofgood}
Recall from Remark \ref{rem:celaxmonoidal} that there is a lax monoidal functor sending a $\mf{g}$-representation $E$ to its Chevalley-Eilenberg complex $C^*(\mf{g}, E)$. This functor is part of an adjoint pair
\begin{equation}\label{diag:koszulquillenpair}\xymatrix{
K(\mf{g})\otimes_{C^*(\mf{g})} (-)\colon \dgMod_{C^*(\mf{g})} \ar@<1ex>[r] & \dgRep_\mf{g}\colon C^*(\mf{g}, -)\ar@<1ex>[l]
}\end{equation}
were $K(\mf{g})$ is the Koszul complex of $\mf{g}$ (see Remark \ref{fm:rem:cofibersequenceoncotangent}). 

When $\mf{g}$ has a graded-projective underlying dg-$A$-module, this adjoint pair is a Quillen pair between the \emph{projective} model structure on dg-$C^*(\mf{g})$-modules and the \emph{tame} model structure on $\mf{g}$-representations, transferred from the tame model structure on dg-$A$-modules. We let $\Phi_\mf{g}\colon \Mod_{C^*(\mf{g})}\rt \Rep^!_\mf{g}$ denote the left derived functor between $\infty$-categories. One can think of $\Phi_\mf{g}$ as the functor $F\mapsto A\otimes_{C^*(\mf{g})} F$, since $K(\mf{g})$ is equivalent to the $\mf{g}$-representation $A$.
\begin{remark}
When the underlying dg-$A$-module of $\mf{g}$ is projectively cofibrant, the adjoint pair \eqref{diag:koszulquillenpair} is also a Quillen pair for the projective model structure on dg-$A$-modules. In terms of $\infty$-categories, this means that the functor $\Phi_\mf{g}$ factors as
$$\xymatrix{
\Phi_\mf{g}\colon \Mod_{C^*(\mf{g})}\ar[r] & \Rep_\mf{g}\ar@{^{(}->}[r]& \Rep_\mf{g}^!.
}$$
\end{remark}
\begin{lemma}\label{fm:lem:fullyfaithfulonqcoh}
Suppose that $\mf{g}\in \LieAlgd_A^!$ is compact. Then $\Phi_\mf{g}\colon \Mod_{C^*(\mf{g})}\rt \Rep^!_{\mf{g}}$ is fully faithful.
\end{lemma}
\begin{proof}
Let $\mc{C}$ be the class of objects in $\Mod_{C^*(\mf{g})}$ for which the derived unit map is an equivalence. Then $\mc{C}$ is closed under finite colimits and retracts and contains $C^*(\mf{g})$ because $\Phi_\mf{g}(C^*(\mf{g}))\simeq A$. It follows that the unit map is an equivalence for all compact $C^*(\mf{g})$-modules. Since $\Mod_{C^*(\mf{g})}\simeq \mm{Ind}(\Mod_{C^*(\mf{g})}^\omega)$ is the ind-completion of the category of compact $C^*(\mf{g})$-modules, it suffices to show that the right adjoint $C^*(\mf{g}, -)$ preserves filtered colimits. Equivalently, it suffices to verify that $\Phi_\mf{g}(C^*(\mf{g}))\simeq A$ is a compact object of $\Rep_\mf{g}$.

By Remark \ref{fm:rem:cofibersequenceoncotangent}, $A$ fits into a cofiber sequence of $\mc{U}(\mf{g})$-modules
$$\xymatrix{
L_\mf{g}\ar[r] & \mc{U}(\mf{g})\ar[r] & A.
}$$
It suffices to show that $L_\mf{g}$ is compact. This follows from the fact that the cotangent complex functor preserves compact objects, since its right adjoint (taking square zero extensions) preserves filtered colimits, which are computed at the level of tame dg-$A$-modules (Proposition \ref{prop:monadicitytame}).
\end{proof}
\begin{lemma}\label{fm:lem:generatorsforconnectivemodules}
Suppose that $\mf{g}\in \LieAlgd_A^!$ admits a finite filtration 
$$\xymatrix{
0=\mf{g}^{(0)}\ar[r] & \dots \ar[r] & \mf{g}^{(n)}=\mf{g}
}$$ 
with the following property: for each $i$, there is a dg-$A$-module $V$ contained in $\cat{T}^{\sleq -2}$ (see Lemma \ref{lem:tamecofgenerated}), such that $\mf{g}^{(i-1)}\rt \mf{g}^{(i)}$ is a pushout of the map $F(V)\rt 0$ from the free dg-Lie algebroid on $0\colon V\rt T_A$.

Let $E$ be a dg-$\mf{g}$-representation whose underlying tame dg-$A$-module is connective (Remark \ref{rem:connectivemodules}). Then there exists a map $\bigoplus_\alpha A \rt E$ in the $\infty$-category $\Rep^!_\mf{g}$ which induces a surjection on $\pi_0$.
\end{lemma}
\begin{proof}
Pick representatives $e_\alpha\in E$ for the generators of $\pi_0(E)$ and consider the associated map of $\mf{g}$-representations $\bigoplus_\alpha \mc{U}(\mf{g})\rt E$. This map is clearly surjective on $\pi_0$, so it suffices to prove that it factors (up to homotopy) as
$$\xymatrix{
\bigoplus_\alpha \mc{U}(\mf{g})\ar[r] & \bigoplus_\alpha A\ar[r] & E. 
}$$
Using the cofiber sequence $L_\mf{g}\rt \mc{U}(\mf{g})\rt A$, we therefore have to provide a null-homotopy of each composite map
$$\xymatrix{
L_\mf{g}\ar[r] & \mc{U}(\mf{g})\ar[r]^{e_\alpha} & E.
}$$
By the assumption on $\mf{g}$, the cotangent complex $L_\mf{g}$ admits a filtration by $\mc{U}(\mf{g})$-modules
$$\xymatrix{
0=L_\mf{g}^{(0)}\ar[r] & \dots \ar[r] & L_\mf{g}^{(n)}=L_\mf{g}
}$$
where each $L_\mf{g}^{(i-1)}\rt L_\mf{g}^{(i)}$ has cofiber of the form $\mc{U}(\mf{g})\otimes_A V$, with $V\in \cat{T}^{\sleq -1}$. An inductive application of Remark \ref{rem:connectivemodules} now shows that any map $L_\mf{g}\rt E$ to a connective $\mf{g}$-representation is null-homotopic, which concludes the proof.
\end{proof}
\begin{corollary}\label{fm:cor:equivonconnmodules}
Suppose that $A$ is eventually coconnective and let $\mf{g}\in\LieAlgd_A$ be good. Then $\Phi_\mf{g}\colon\Mod_{C^*(\mf{g})}\rt \Rep_\mf{g}$ is fully faithful and induces an equivalence
$$\xymatrix{
\Phi_\mf{g}\colon \Mod_{C^*(\mf{g})}^{\sgeq 0} \ar[r] & \Rep_{\mf{g}}^{\sgeq 0}
}$$
between the full subcategories consisting of connective $C^*(\mf{g})$-modules and representations whose underlying $A$-module is connective.
\end{corollary}
\begin{corollary}\label{fm:cor:equivonconnprocoherent}
Suppose that $A$ is bounded coherent and let $\mf{g}\in\LieAlgd_A^!$ be good. Then $\Phi_\mf{g}\colon\Mod_{C^*(\mf{g})}\rt \Rep_\mf{g}$ is fully faithful and induces an equivalence
$$\xymatrix{
\Phi_\mf{g}\colon \Mod_{C^*(\mf{g})}^{\sgeq 0} \ar[r] & \Rep_{\mf{g}}^{!,\sgeq 0}
}$$
between the full subcategories consisting of connective $C^*(\mf{g})$-modules and representations whose underlying pro-coherent sheaf over $A$ is connective.
\end{corollary}
\begin{proof}
Both assertions are proven in the same way. Since $\mf{g}$ is good, it is compact in $\LieAlgd_A^!$ (or $\LieAlgd_A$), so that $\Phi_\mf{g}$ is fully faithful by Lemma \ref{fm:lem:fullyfaithfulonqcoh}.

The subcategory $\Mod_{C^*(\mf{g})}^{\sgeq 0}\subseteq \Mod_{C^*(\mf{g})}$ is the smallest subcategory of $\Mod_{C^*(\mf{g})}$ which is closed under colimits and extensions and which contains $C^*(\mf{g})$. As a consequence, the essential image 
$$
\cat{C}:= \Phi_\mf{g}\Big(\Mod_{C^*(\mf{g})}^{\sgeq 0}\Big)\subseteq \Rep_\mf{g}^!
$$ under $\Phi_\mf{g}$ is the smallest subcategory of $\Rep^!_\mf{g}$ which is closed under colimits and extensions and which contains $A$. Clearly $\cat{C}$ is contained in $\Rep_\mf{g}^{!, \sgeq 0}$, so it suffices to prove the reverse inclusion. 

To this end, let $E$ be a left $\mc{U}(\mf{g})$-module whose underlying (tame) dg-$A$-module is connective. We will inductively construct a sequence of left $\mc{U}(\mf{g})$-modules $0=E^{(-1)}\rt E^{(0)}\rt \cdots \rt E$ such that each $E^{(n)}\in \cat{C}$ and such that each map $E^{(n)}\rt E$ induces an isomorphism on homotopy groups in degrees $<n$ and a surjection on $\pi_n$. It follows that the map $\colim E^{(n)}\rt E$ is a weak equivalence, so that $E\in \cat{C}$. 

To construct this sequence, suppose we have constructed $E^{(n-1)}$ and let $F$ be the fiber of the map $E^{(n-1)}\rt E$. Then $F$ is a left $\mc{U}(\mf{g})$-module whose underlying (tame) dg-$A$-module is $(n-2)$-connective. Since $\mf{g}$ is good, a shift of Lemma \ref{fm:lem:generatorsforconnectivemodules} shows that there exists a map $\bigoplus_\alpha A[n-2]\rt F$ which induces a surjection on $\pi_{n-2}$. Now let $E^{(n)}$ be the cofiber of the map $\bigoplus_\alpha A[n-2]\rt F\rt E^{(n-1)}$. This cofiber is contained in $\cat{C}$ and the five lemma shows that $E^{(n)}\rt E$ induces an isomorphism on homotopy groups in degrees $<n$ and a surjection on $\pi_n$.
\end{proof}

\subsection{Naturality}
Let us now address the functoriality of $\Phi_\mf{g}$ in the dg-Lie algebroid $\mf{g}$. This is somewhat delicate, because the Quillen adjunction
$$\xymatrix{
K(\mf{g})\otimes_{C^*(\mf{g})} (-)\colon \dgMod_{C^*(\mf{g})} \ar@<1ex>[r] & \dgRep_\mf{g}\colon C^*(\mf{g}, -)\ar@<1ex>[l]
}$$
does not strictly intertwine restriction of $\mf{g}$-representations with induction of $C^*(\mf{g})$-modules. However, this does become true at the level of $\infty$-categories. To see this, let us make the following definitions:
\begin{construction}\label{fm:con:smcatofdgmodules}
Let $\Mod^{\dg, \otimes}$ be the category in which
\begin{smitemize}
 \item an object is a tuple $(B, M_1, \dots, M_m)$, where $B$ is a cdga and each $M_i$ is a dg-$B$-module.
 \item a morphism $(B, M_1, \dots, M_m)\rt (C, N_1, \dots, N_n)$ is a map of finite pointed sets $\alpha\colon \big<m\big>\rt \big<n\big>$, a map of cdgas $B\rt C$ and $B$-multinear maps $\bigotimes^{\alpha(i)=j} M_i\rt N_j$.
\end{smitemize}
Similarly, let $\Rep^{\dg, \otimes}$ be the category in which
\begin{smitemize}
 \item an object is a tuple $(\mf{g}, E_1, \dots, E_m)$, where $\mf{g}$ is a (tame) $A$-cofibrant dg-Lie algebroid and each $E_i$ is a $\mf{g}$-representation.
 \item a morphism $(\mf{g}, E_1, \dots, E_m)\rt (\mf{h}, F_1, \dots, F_n)$ is a map of finite pointed sets $\alpha\colon \big<m\big>\rt \big<n\big>$, a map of dg-Lie algebroids $\mf{h}\rt \mf{g}$ and maps of $\mf{h}$-representations $\bigotimes^{\alpha(i)=j}_A E_i\rt F_j$.
 \end{smitemize}
These categories fit into a commuting square
\begin{equation}\label{diag:ceatdglevel}\vcenter{\xymatrix@R=1.7pc{
\Rep^{\dg, \otimes}\ar[d]\ar[r]^-{C^*} & \Mod^{\dg, \otimes}\ar[d]\\
\left(\LieAlgd_A^{\dg, A-\mm{cof}}\right)^\op\times \Gamma^\op\ar[r]_-{C^*} & \dgCAlg_k\times \Gamma^\op.
}}\end{equation}
The vertical functors are the obvious projections, which are both cocartesian fibrations. The top functor sends $(\mf{g}, E_1, \dots, E_m)$ to $\big(C^*(\mf{g}), C^*(\mf{g}, E_1), \dots, C^*(\mf{g}, E_m)\big)$. 
\end{construction}
\begin{lemma}
After inverting the tame weak equivalences on the left and quasi-isomorphisms on the right of \eqref{diag:ceatdglevel}, one obtains a commuting square of $\infty$-categories
$$\xymatrix@R=1.7pc{
\Rep^{!, \otimes}\ar[r]^{C^*}\ar[d] & \Mod^\otimes\ar[d]\\
\big(\LieAlgd_A^{!}\big)^{\op}\times\Gamma^\op\ar[r]_-{C^*} & \CAlg_k\times \Gamma^\op.
}$$
in which the vertical functors are cocartesian fibrations.
\end{lemma}
\begin{proof}
All functors in \eqref{diag:ceatdglevel} preserve weak equivalences, so that they descend to functors between localizations. It suffices to verify that the vertical projections remain cocartesian fibrations after inverting the quasi-isomorphisms. To see this, let $\cat{C}\subseteq \Rep^{\dg, \otimes}$ be the full subcategory on $(\mf{g}, E_1, \dots, E_m)$ where each $E_i$ is tamely cofibrant as a dg-$A$-module. The inclusion $\cat{C}\subseteq \Rep^{\dg, \otimes}$ induces an equivalence on localizations, with inverse provided by a cofibrant replacement functor. The projection
$$\xymatrix{
\cat{C}\ar[r] & \Big(\LieAlgd_A^{\dg, A-\mm{cof}}\Big)^\op\times \Gamma^\op
}$$
is a cocartesian fibration. For any map $\alpha\colon \big<m\big>\rt \big<n\big>$ and any $f\colon \mf{h}\rt \mf{g}$, the induced functor between fibers is given by
$$\xymatrix{
\left(\Rep_\mf{g}^{\dg, A-\mm{cof}}\right)^{\times m}\ar[r] & \left(\Rep_\mf{h}^{\dg, A-\mm{cof}}\right)^{\times n}; \hspace{4pt} (E_j)_{j\leq m}\ar[r] & \left(\bigotimes_{\alpha(j)=i} f^*E_j\right)_{i\leq m}.
}$$
This functor preserves all quasi-isomorphisms and induces an equivalence of $\infty$-categories whenever $\alpha$ is a bijection and $f$ is a quasi-isomorphism. It follows from \cite[Proposition 2.1.4]{hin16} that the induced functor of $\infty$-categories
$$\smash{\xymatrix{
\Rep^\otimes\ar[r] & \LieAlgd_A^\op\times\Gamma^\op
}}$$
is a cocartesian fibration. A similar argument shows that $\Mod^\otimes\rt \CAlg_k\times\Gamma^\op$ is a cocartesian fibration.
\end{proof}
\begin{remark}\label{rem:cocartesianedges}
Fix a map $\alpha\colon \big<m\big>\rt \big<n\big>$, a map $f\colon \mf{h}\rt \mf{g}$ and a collection $E_1, \dots, E_m$ of $\mf{g}$-representations that are tamely cofibrant dg-$A$-modules. By \cite[Proposition 2.1.4]{hin16}, the cocartesian lift of $(f, \alpha)$ in $\Rep^\otimes$ with domain $(\mf{g}, E_1, \dots, E_m)$ is the image of
$$\xymatrix{
(\mf{g}, E_1, \dots, E_m)\ar[r] & \left(\mf{h}, \bigotimes_{\alpha(j)=1} f^*E_j, \dots, \bigotimes_{\alpha(j)=n} f^*E_j\right)
}$$
in the $\infty$-categorical localization of $\Rep_\mf{g}^{\dg, \otimes}$.
\end{remark}
\begin{lemma}\label{fm:lem:functorialleftadjoint}
Let $\smash{\Mod^\otimes_{C^*}\rt \big(\LieAlgd^!_A\big){}^\op\times\Gamma^\op}$ denote the base change of the projection $\Mod^\otimes \rt \CAlg_k\times\Gamma^\op$ along the functor $C^*$, so that there is a functor of cocartesian fibrations
$$\xymatrix@C=1.6pc@R=1.6pc{
\Rep^{!, \otimes}\ar[rr]^{C^*}\ar[rd] & & \Mod^\otimes_{C^*}\ar[ld]\\
& \big(\LieAlgd^!_A\big)^\op\times\Gamma^\op. & 
}$$
This functor admits a left adjoint $\Phi$, which preserves cocartesian edges.
\end{lemma}
\begin{proof}
For each $\big<m\big>$ and each dg-Lie algebroid $\mf{g}$, the functor between the fibers 
$$\smash{\xymatrix{
C^*\colon \big(\Rep^!_\mf{g}\big){}^{\times m}\ar[r] & \Mod^{\times m}_{C^*(\mf{g})}
}}$$
admits a left adjoint $\Phi_\mf{g}$: this is just the $m$-fold product of the left derived functor of the Quillen pair \eqref{diag:koszulquillenpair}. By Remark \ref{rem:cocartesianedges} and \cite[Proposition 7.3.2.11]{lur16}, the existence of the global left adjoint $\Phi\colon \Mod_{C^*}\rt \Rep$, as well as the fact that it preserves cocartesian edges, follows once we verify the following: for any map of dg-Lie algebroids $f\colon \mf{h}\rt \mf{g}$ and any collection of $C^*(\mf{g})$-modules $M_i$, the natural map
$$\xymatrix{
\Phi_\mf{h}\big(C^*(\mf{h})\otimes_{C^*(\mf{g})}M_1\otimes_{C^*(\mf{g})}\dots\otimes_{C^*(\mf{g})} M_m\big)\ar[r] & \Phi_{\mf{g}}(M_1)\otimes_A \dots \otimes_A \Phi_\mf{g}(M_m)
}$$
is an equivalence. Since both functors preserve colimits of modules in each variable, we can reduce to the case where each $M_i$ is equivalent to $C^*(\mf{g})$. In that case, the map can be identified with a map
$$\smash{\xymatrix{
K(\mf{h})\ar[r] & K(\mf{g})\otimes_A \dots \otimes_A K(\mf{g})
}}$$
between Koszul complexes. Since both $K(\mf{g})$ and $K(\mf{h})$ were resolutions of the canonical representation $A$, the result follows.
\end{proof}
In other words, the functors $\Phi_\mf{g}$ from Section \ref{fm:sec:representationsofgood} determine a natural (symmetric monoidal) transformation between diagrams of symmetric monoidal $\infty$-categories.

\subsection{Quasicoherent sheaves}\label{fm:sec:qcoh}
When restricted to $\LieAlgd_A\subseteq \LieAlgd_A^!$, the left adjoint $\Phi$ of Lemma \ref{fm:lem:functorialleftadjoint} corresponds under straightening to a natural transformation
$$\xymatrix@C=3pc{
\LieAlgd_A^\op\ar@/^1.5pc/[r]^{\Mod_{C^*}}_{}="s" \ar@/_1.5pc/[r]_{\Rep}^{}="t" & \hspace{50pt}\cat{Pr_{sym. mon.}^\mm{L}}=\CAlg(\cat{Pr^L}) \ar@{=>}"s";"t"^{\Phi}
}$$
between diagrams of presentable (closed) symmetric monoidal $\infty$-categories, with symmetric monoidal left adjoint functors between them. This natural transformation is given pointwise by the left derived functor $\Phi_\mf{g}$ of \eqref{diag:koszulquillenpair}.

We can precompose with the duality functor $\mf{D}$ \eqref{diag:deformationadjunction} and obtain a natural transformation of functors
\begin{equation}\label{fm:diag:psinattransformation}\vcenter{\xymatrix{
\CAlg_k^\mm{sm}/A\ar[d]\ar[rr] \ar[d] & & \CAlg_k\ar[d]^{\Mod}\\
\Fun(\CAlg_k^\mm{sm}/A, \sS)^\op\ar[r]_-{\mf{D}_!} & \LieAlgd_A^\op\ar[r]_-{\Rep} & \cat{Pr_{sym. mon.}^\mm{L}} \ar@{=>}[-1,0]+<-50pt, -10pt>;[0, -2]+<50pt, 15pt>^\Psi
}}\end{equation}
Here $\mf{D}_!$ is the unique limit-preserving functor which restricts to $\mf{D}$ on the corepresentable functors. For each $B$ in $\CAlg_k^\mm{sm}/A$, the functor $\Psi_B$ is the composite
\begin{equation}\label{fm:diag:compositeformodules}\vcenter{\xymatrix{
\Psi_{B}\colon \Mod_{B}\ar[r] & \Mod_{C^*\mf{D}(B)} \ar[r]^{\Phi_{\mf{D}(B)}} & \Rep_{\mf{D}(B)}
}}\end{equation}
where the first functor arises from the unit map $B\rt C^*\mf{D}(B)$.

\begin{definition}\label{def:qc}
The presentable, symmetric monoidal $\infty$-category $\QC(X)$ of \emph{quasi-coherent sheaves} on a functor $X\colon \CAlg_k^\mm{sm}/A\rt \sS$ is the value on $X$ of the right Kan extension
$$\xymatrix{
\CAlg_k^\mm{sm}/A\ar[r]^\Mod\ar[d] & \cat{Pr_{sym. mon.}^\mm{L}}\\
\Fun(\CAlg_k^\mm{sm}/A, \sS)^\op\ar@{..>}[ru]_{\QC}
}$$
This right Kan extension exists by \cite[Lemma 5.1.5.5]{lur09}, since $ \cat{Pr_{sym. mon.}^\mm{L}}$ has all small limits \cite[Proposition 4.8.1.15]{lur16}. 
\end{definition}
By the universal property of $\QC$, we obtain a natural transformation of presentable symmetric monoidal $\infty$-categories $\Psi_X\colon\QC(X)\rt \Rep_{\mf{D}_!(X)}$. When $X$ is corepresentable, this is simply the functor $\Psi$ from \eqref{fm:diag:psinattransformation}. If $A$ is eventually coconnective and $X$ is a formal moduli problem over $A$, then $\mf{D}_!(X)$ is naturally equivalent to $T_{A/X}$, by Remark \ref{rem:mcintermsofD}. We thus obtain a natural symmetric monoidal functor
\begin{equation}\label{fm:diag:psi_X}\vcenter{\xymatrix{
\Psi_X\colon \QC(X)\ar[r] & \Rep_{T_{A/X}}
}}\end{equation}
for any formal moduli problem $X$. By naturality, the composition
$$\smash{\xymatrix{
\QC(X)\ar[r]^{\Psi_X} \ar[r] & \Rep_{T_{A/X}} \ar[r] & \Rep_0 = \Mod_A 
}}$$
is naturally equivalent to the functor $x^*\colon \QC(X)\rt \Mod_A$ that restricts $F$ to the basepoint $x\in X(A)$.
\begin{proof}[Proof (of Theorem \ref{fm:thm:quasi-coherentvsrepresentations})]
We have to prove that for any formal moduli problem $X$, $\Psi_X$ \eqref{fm:diag:psi_X} is fully faithful and restricts to an equivalence between $\QC^{\sgeq 0}(X)$ and $\Rep_{T_{A/X}}^{\sgeq 0}$. When $X$ is representable by an object $B\in \cat{CAlg}_k^\mm{sm}/A$, this functor is given by the composite \eqref{fm:diag:compositeformodules}. The first functor is an equivalence by Theorem \ref{thm:formalmoduliperfect}, so that the result follows from Corollary \ref{fm:cor:equivonconnmodules}.
 
The functor $\QC$ sends sifted colimits in $\cat{Fun}(\cat{CAlg}^\mm{sm}/A, \sS)$ to limits of symmetric monoidal $\infty$-categories by construction. Lemma \ref{fm:lem:algrepspreservessiftedcolim} implies that the same assertion holds for $X\mapsto \Rep_{\mf{D}_!(X)}$. Every formal moduli problem is a sifted colimit of representable functors, since every Lie algebroid is a sifted colimit of good Lie algebroids. It follows that $\Psi_X$ is fully faithful, being a limit of fully faithful functors. 

It remains to identify the essential image of $\QC(X)^{\sgeq 0}$. Recall that a quasi-coherent sheaf $F$ is connective if and only if $f^*F\in \Mod_{B}$ is connective for any $f\in X(B)$. In terms of its image $\Psi_X(E)$, this means for any map $f\colon \mf{D}(B)\rt T_{A/X}$, the restricted representation $f^!\Psi_X(E)$ is the image of a connective $B$-module. By Corollary \ref{fm:cor:equivonconnmodules}, this is equivalent to $\Psi_X(E)$ being a connective $T_{A/X}$-representation.
\end{proof}
\begin{corollary}
Let $A$ be eventually coconnective and let $X$ be a formal moduli problem such that $T_{A/X}$ is connective. Then there is an equivalence
$$\xymatrix{
\Psi_X\colon \QC(X)\ar[r] & \Rep_{T_{A/X}}.
}$$
\end{corollary}
\begin{proof}
Since $T_{A/X}$ is connective, its enveloping algebra $\mc{U}(T_{A/X})$ is connective as well. It follows that $T_A$-representations carry a \emph{right complete} $t$-structure, where $\Rep^{\sgeq 0}_{T_{A/X}}$ consists of $\mf{g}$-representations whose underlying chain complex is connective \cite[Example 2.2.1.3]{lur16}. Similarly, $\QC(X)$ carries a right complete $t$-structure where $F\in \QC(X)^{\sgeq 0}$ if and only if $f^*F$ is a connective chain complex for all $f\in X(B)$. 

The functor $\Psi_X$ fits into a sequence of locally presentable $\infty$-categories and left adjoint functors between them
$$\xymatrix{
\QC(X)^{\sgeq 0}\ar[d]_{\Psi_X^{\sgeq 0}}\ar[r] & \QC(X)^{\sgeq -1}\ar[r]\ar[d]_{\Psi_X^{\sgeq -1}} & \dots \ar[r] & \QC(X)\ar[d]^{\Psi_X}\\
\Rep_{T_{A/X}}^{\sgeq 0}\ar[r] & \Rep_{T_{A/X}}^{\sgeq -1}\ar[r] & \dots \ar[r] &  \Rep_{T_{A/X}}.
}$$
Since $\QC(X)$ and $\Rep_{T_{A/X}}$ are right complete, the horizontal sequences are colimit diagrams, so that the result follows Theorem \ref{fm:thm:quasi-coherentvsrepresentations}.
\end{proof}

\subsection{Deformations of algebras}
Theorem \ref{fm:thm:quasi-coherentvsrepresentations} can be used to study the deformation theory of connective (commutative) $A$-algebras. Suppose that $R$ is a cofibrant commutative dg-algebra in $\Mod^{\dg, \sgeq 0}_A$, so that $R$ determines an object in the $\infty$-category $\CAlg(\Mod^{\sgeq 0}_A)$. Consider the functor
$$\xymatrix{
\mm{Def}_R\colon \cat{FormMod}_A^\op\ar[r] & \widehat{\cat{Cat}}_\infty; \hspace{4pt} X\ar@{|->}[r] & \CAlg(\QC(X))\times_{\CAlg(\Mod_A)}\{R\}
}$$
sending each formal moduli problem $X$ to the (locally small) $\infty$-category of commutative algebras in $\QC(X)$, equipped with an equivalence between $R$ and their restriction to the canonical basepoint $x\in X(A)$. One can think of a point in $\mm{Def}_R(X)$ as a \emph{deformation} of the commutative algebra along the map of stacks $\spec(A)\rt X$.

Every such deformation is necessarily connective, since $R$ itself is connective. It follows from Theorem \ref{thm:formalmoduliperfect} and Theorem \ref{fm:thm:quasi-coherentvsrepresentations} that this functor can be identified with the functor
$$\smash{\xymatrix{
\mm{Act}_R\colon \LieAlgd_A^\op\ar[r] & \widehat{\cat{Cat}}_\infty; \hspace{4pt} \mf{g}\ar@{|->}[r] & \CAlg(\Rep_\mf{g})\times_{\CAlg(\Mod_A)} \{R\}
}}$$
sending each Lie algebroid $\mf{g}$ to the category of commutative algebras in $\Rep_\mf{g}$ whose underlying commutative $A$-algebra is equivalent to $R$. In light of Example \ref{ex:algebrainrep}, one an think of $\mm{Act}_R(\mf{g})$ as the category of action of $\mf{g}$ on $R$ by derivations.

It follows from Lemma \ref{fm:lem:algrepspreservessiftedcolim} that $\mm{Act}_R$, and therefore $\mm{Def}_R$, preserves all limits. In particular, the restriction of $\mm{Def}_R$ along the Yoneda embedding gives rise to a formal moduli problem
$$\xymatrix{
\mm{Def}_R\colon \CAlg_k^\mm{sm}/A\ar[r] & \sS; \hspace{4pt} B\ar@{|->}[r] & \CAlg(\Mod_B)\times_{\CAlg(\Mod_A)} \{R\}.
}$$
This takes values in (small) spaces because $\mm{Def}_R(A\oplus A[n])\simeq \Omega \mm{Def}_R(A\oplus A[n+1])$ and all $\infty$-categories involved are locally small. In particular, $\mm{Def}_R$ is classified by a certain Lie algebroid $\mf{h}$.

Recall from Example \ref{ex:atiyah2} that associated to the commutative dg-algebra $R$ in $\dgMod_A$ is a dg-Lie algebroid $\mm{At}_{\field{E}_\infty}(R)$ of compatible derivations of $A$ and $R$. Since $R$ is a cofibrant commutative dg-algebra, $\mm{At}_{\field{E}_\infty}(R)$ is a fibrant dg-Lie algebroid. There is an obvious representation of $\mm{At}_{\field{E}_\infty}(R)$ on $R$ itself by means of derivations, which is classified by a map
$$\smash{\xymatrix{
\phi\colon \mm{At}_{\field{E}_\infty}(R)\ar[r] & \mf{h}.
}}$$
We have the following (folklore) result, cf.\ \cite{toe16} (see also \cite{hin04}):
\begin{proposition}\label{prop:defofalgebras}
The map $\phi$ is an equivalence for any cofibrant connective dg-$A$-algebra $R$. In other words, $\mm{At}_{\field{E}_\infty}(R)$ classifies the formal moduli problem $\mm{Def}_R$.
\end{proposition}
\begin{proof}
Let $\mf{s}_n$ be the free Lie algebroid on the map $0\colon A[n]\rt T_A$, with natural maps $i\colon 0\rt \mf{s}_n$ and $r\colon \mf{s}_n\rt 0$. It suffices to verify that the maps $\phi_*$ in the following commuting diagram are bijections for all $n$:
$$\xymatrix@R=1.9pc{
\pi_1\Map\big(\mf{s}_{n}, \mm{At}_{\field{E}_\infty}(R)\big) \ar[r]^{\phi_*}\ar[d]^{\cong} & \pi_1\Map\big(\mf{s}_{n}, \mf{h}\big) \ar[r]^-{\cong}\ar[d]^{\cong} & \pi_1\mm{Act}_R(\mf{s}_{n})=: M_1\ar[d]_\psi^{\cong}\\
\pi_0\Map(\mf{s}_{n+1}, \mm{At}_{\field{E}_\infty}(R))\ar[r]_{\phi_*} & \pi_0\Map\big(\mf{s}_{n+1}, \mf{h}\big) \ar[r]_-{\cong} & \pi_0\mm{Act}_R(\mf{s}_{n+1})=: M_0.
}$$
The dg-Lie algebroid $\mm{At}(R)$ is fibrant and the kernel of its anchor map is the Lie algebra of $A$-linear derivations of $R$. The abelian group $\pi_1\Map\big(\mf{s}_{n}, \mm{At}_{\field{E}_\infty}(R)\big)$ can therefore be identified with $\pi_{n+1}\mm{Der}_A(R, R)$.

On the other hand, $M_1$ can \emph{also} be identified with $\pi_{n+1}\mm{Der}_A(R, R)$. To see this, let $r^!R\in \CAlg(\Rep_{\mf{s}_n})$ be the trivial representation of $\mf{s}_{n}$ on $R$. We can identify $M_1$ with the set of homotopy classes of maps $\alpha\colon r^!R\rt r^!R$ such that the restriction $i^!(\alpha)$ is the identity map in $\CAlg(\Mod_A)$. Using that the restriction functors $r^!$ and $i^!$ have a right adjoints $r_!$ and $i_!$, such maps can be identified with sections of
$$\xymatrix{
r_!r^!(R)\ar[r]^-{r_!\eta r^!} & r_!i_!i^!r^!(R) = R
}$$
in the $\infty$-category $\CAlg(\Mod_A)$. The right adjoint $r_!$ can be computed as
$$
r_!(r^!R) \simeq \Hom^\field{R}_{\mc{U}(\mf{s}_n)}(A, r^!R) \simeq C^*\big(\mf{s}_n, r^!R\big).
$$
In fact, since $\mf{s}_n$ is a free Lie algebroid, the discussion of Section \ref{fm:sec:ceoffree} shows that restriction along $A[n]\rt\mf{s}_n$ induces a weak equivalence of commutative dg-$A$-algebras
$$\xymatrix{
C^*(\mf{s}_n, r^!R)\ar[r]^-\sim & R\oplus \Hom_A\big(A[n+1], r^!R\big) \cong R\oplus R[-n-1]
}$$
to the \emph{square zero extension} of $R$ by a shifted copy of itself. We conclude that $M_1$ is isomorphic to the set of homotopy classes of sections in $\CAlg(\Mod_A)$ of
$$\smash{\xymatrix{
R\oplus R[-n-1]\ar[r] & R.
}}$$
The set of such sections is indeed isomorphic to $\pi_{n+1}\mm{Der}_A(R, R)$, as asserted. 

Consider an element $v\in\mm{Der}_A(R, R)\cong M_1$, corresponding to a map $v\colon r^!R\rt r^!R$ in $\CAlg(\Rep_{\mf{s}_n})$. Unwinding the definitions, the image $\psi(v)\in M_0$ is given by (the equivalence class of) the $\mf{s}_{n+1}$-action on $R$ where the generator of $\mf{s}_{n+1}$ acts by $v$. This representation is exactly the image of the map $v\colon A[n+1]\rt \mf{s}_{n+1}\rt \mm{At}(R)$ under the map $\phi_*$. We conclude that $\phi_*$ is indeed an isomorphism.
\end{proof}
\begin{remark}
The same proof shows that for any (simplicial, coloured) operad $\ope{P}$ and a cofibrant $\ope{P}$-algebra $R$, the formal moduli problem $\mm{Def}_R$ of deformations of $R$ is classified by the Atiyah Lie algebroid $\mm{At}_{\ope{P}}(R)$ of Example \ref{ex:atiyah}. For example, this provides an explicit dg-Lie algebroid classifying the deformations of modules, associative algebras or diagrams thereof.
\end{remark}

\section{Representations and pro-coherent sheaves}\label{fm:sec:indcoh}
Let $X$ be a formal moduli problem over $A$ with associated Lie algebroid $T_{A/X}$. We have seen in the previous section that the $\infty$-category of $T_{A/X}$-representation is an \emph{extension} of the $\infty$-category of quasi-coherent sheaves on $X$. The purpose of this section is to provide a geometric description of this extension when $A$ is bounded coherent, by identifying it with the $\infty$-category of \emph{pro-coherent sheaves} on $X$ (Definition \ref{def:procohonformmod}).
\begin{theorem}\label{fm:thm:ind-coherentvsrepresentations}
Let $A$ be a cofibrant, bounded and coherent commutative dg-$k$-algebra and let $X$ be a pro-coherent formal moduli problem over $A$. Then there is an equivalence
$$\xymatrix{
\Psi_X\colon \QC^!(X)\ar[r]^-\sim & \Rep^!_{T_{A/X}}
}$$
between the $\infty$-categories of \emph{pro-coherent sheaves} on $X$ and pro-coherent $T_{A/X}$-representations. For any map $f\colon X\rt Y$ of pro-coherent formal moduli problems, this equivalence identifies the functor $f^!\colon \QC^!(Y)\rt \QC^!(X)$ with the restriction functor from $T_{A/Y}$-representations to $T_{A/X}$-representations.
\end{theorem}
\begin{remark}
Suppose that $A$ is coherent and eventually coconnective and let $X\in \cat{FormMod}_A$ be an ordinary formal moduli problem over $A$, with Lie algebroid $T_{A/X}$. The $\infty$-category of (quasi-coherent) $T_{A/X}$-representations can then be identified with the full subcategory of pro-coherent sheaves on $X$ whose restriction to $A$ is quasi-coherent.
\end{remark}
We will begin by discussing pro-coherent representations over good pro-coherent Lie algebroids in Section \ref{sec:koszulformodules}. In Section \ref{sec:indcohonformmod} we use this to prove Theorem \ref{fm:thm:ind-coherentvsrepresentations}.

\subsection{Koszul duality for modules}\label{sec:koszulformodules}
Recall that the left adjoint $\Phi$ from Lemma \ref{fm:lem:functorialleftadjoint} corresponds under straightening to a natural transformation
between diagrams
$$\smash{\xymatrix{
\Rep^!, \Mod_{C^*}\colon \LieAlgd_A^{!, \op}\ar[r] & \cat{Pr}^\mm{L}
}}$$
of locally presentable $\infty$-categories and left adjoint functors between them. When evaluated on an arrow $f\colon \mf{g}\rt \mf{h}$, this natural transformation is given by the commuting square of left adjoints
$$\hspace{-60pt}\xymatrix{
\Mod_{C^*(\mf{h})}\ar[d]_{f^*:=C^*(\mf{g})\otimes_{C^*(\mf{h})}-} \ar[r]^-\Phi & \Rep^!_\mf{h}\ar[d]^{f^!}\\
\Mod_{C^*(\mf{g})}\ar[r]_-\Phi & \Rep^!_\mf{g}. 
 }$$ 
Let us pass to the associated diagram of right adjoint functors and next take opposite categories. The natural transformation $\Phi$ then determines a natural transformation between two diagrams of large $\infty$-categories and left adjoint functors between them
$$\smash{\xymatrix{
\Rep^{!, \op}\colon \LieAlgd_A^{!, \op}\ar[r] & \widehat{\cat{Cat}}{}^\mm{L}_\infty & \Mod_{C^*}^\op\colon \LieAlgd_A^{!, \op}\ar[r] & \widehat{\cat{Cat}}{}^\mm{L}_\infty.
}}$$ 
The value of this natural transformation on an arrow $f\colon \mf{g}\rt \mf{h}$ is given by
\begin{equation}\label{diag:coinduction}\vcenter{\xymatrix@C=3pc{
\Rep^{!,\op}_\mf{g}\ar[d]_{f_!}\ar[r]^-{C^*(\mf{g}, -)} & \Mod_{C^*(\mf{g})}^{\op}\ar[d]^{f_*}\\
\Rep^{!,\op}_\mf{h}\ar[r]_-{C^*(\mf{h}, -)} & \Mod_{C^*(\mf{h})}^{\op}
}}\end{equation}
where $f_!$ is the right adjoint to $f^!$, given by coinduction. For every $A$-cofibrant dg-Lie algebroid $\mf{g}$, taking the $A$-linear dual of a $\mf{g}$-representation determines a left Quillen functor between the tame model structures on dg-$\mf{g}$-representations
$$\smash{\xymatrix{
(-)^\vee\colon \Rep_\mf{g}^{\dg}\ar[r] & \Rep_\mf{g}^{\dg, \op}
}}$$
whose right adjoint is given by $(-)^\vee$ as well. For every map $f\colon \mf{g}\rt \mf{h}$, this functor intertwines induction and coinduction, i.e.\ there is a natural commuting diagram
\begin{equation}\label{diag:inductionandcoinduction}\vcenter{\xymatrix@C=2.5pc{
\Rep^!_\mf{g}\ar[r]^-{(-)^\vee}\ar[d]_{f_*} & \Rep_\mf{g}^{!, \op}\ar[d]^{f_!}\\
\Rep^!_\mf{h}\ar[r]_-{(-)^\vee} & \Rep_\mf{h}^{!, \op}
}}\end{equation} 
at the level of $\infty$-categories. We therefore obtain a natural transformation
\begin{equation}\label{diag:mu}\xymatrix{
\mu\colon \Rep^!_* \ar[r]^-{(-)^\vee} & \Rep^{!, \op}\ar[r]^{C^*} & \Mod^{\op}_{C^*}
}\end{equation}
between functors $\LieAlgd^!_A\rt \widehat{\cat{Cat}}{}^\mm{L}_\infty$. Here $\Rep^!_*$ denotes the functor sending a map of Lie algebroids $f\colon \mf{g}\rt \mf{h}$ to the induction functor $f_*\colon \Rep^!_\mf{g}\rt \Rep^!_\mf{h}$.
\begin{proposition}\label{prop:dualityforreps}
Let $\mf{g}\in \LieAlgd_A^!$ be a good pro-coherent Lie algebroid over $A$. Then the left adjoint functor $\mu$ \eqref{diag:mu} restricts to an equivalence
$$\smash{\xymatrix{
\mu\colon \Rep_\mf{g}^{!,\omega}\ar[r] & \cat{Coh}_{C^*(\mf{g})}^\op
}}$$
between the $\infty$-category of compact pro-coherent $\mf{g}$-representations and the opposite of the $\infty$-category of coherent $C^*(\mf{g})$-modules.
\end{proposition}
\begin{lemma}\label{lem:coherentoversmallext}
Let $\mf{g}$ be a good pro-coherent Lie algebroid over $A$. Consider the thick subcategory $\cat{C}\subseteq \Mod_{C^*(\mf{g})}$ generated by the modules of the form $f_*F$, where $F\in \cat{Coh}_A$ is a coherent $A$-module and $f\colon C^*(\mf{g})\rt A$ is the canonical map. Then $\cat{C}$ coincides with $\cat{Coh}_{C^*(\mf{g})}$.
\end{lemma}
\begin{proof}
To see that $\cat{C}\subseteq \cat{Coh}_{C^*(\mf{g})}$, we have to show that for every coherent $A$-module $F$, $f_*F$ is a coherent $C^*(\mf{g})$-module. Since $C^*(\mf{g})\rt A$ is a small extension, the map $\pi_0(C^*(\mf{g}))\rt \pi_0(A)$ is a surjection of coherent commutative rings. It follows that any finitely generated $\pi_0(A)$-module, such as $\pi_n(F)$, is also finitely generated over $\pi_0(C^*(\mf{g}))$. This means that $f_*F$ is indeed a coherent $C^*(\mf{g})$-module.

It remains to be verified that $\cat{C}$ exhausts $\cat{Coh}_{C^*(\mf{g})}$. Using Postnikov towers, one sees that $\cat{Coh}_{C^*(\mf{g})}$ is the thick subcategory generated by the (discrete) finitely presented $\pi_0(C^*(\mf{g}))$-modules. It therefore suffices to show that $\cat{C}$ contains any such a finitely generated $\pi_0(C^*(\mf{g}))$-module $E$. Since $C^*(\mf{g})\rt A$ is a small extension, there is a finite sequence of square zero extensions by finitely generated $\pi_0(A)$-modules
$$\xymatrix{
\pi_0(C^*(\mf{g}))= R_n\ar[r] & R_{n-1}\ar[r] & \dots \ar[r] & R_0=\pi_0(A).
}$$
We prove by induction that all (discrete) finite generated $R_i$-modules are contained in $\cat{C}$. This clearly holds for finitely generated $\pi_0(A)$-modules, which are coherent $A$-modules. Let $R_i\rt R_{i-1}$ be a square zero extension of $R_{i-1}$ by a finitely generated $\pi_0(A)$-module $I$ and let $E$ be a finitely generated $R_i$-module. Then there is an exact sequence 
$$\xymatrix{
I\otimes_{R_{0}} (R_0\otimes_{R_i} E) \ar[r] & E\ar[r] & R_{i-1}\otimes_{R_n} E\ar[r] & 0.
}$$
Since $I\otimes_{R_{0}} (R_0\otimes_{R_i} E)$ and $R_{i-1}\otimes_{R_n} E$ are contained in $\cat{C}$, so is $E$.
\end{proof}
\begin{proof}[Proof (of Proposition \ref{prop:dualityforreps})]
The right adjoint of the functor $\mu$ is given by
$$\xymatrix{
\nu\colon \Mod_{C^*(\mf{g})}^\op\ar[r] & \Rep_\mf{g}; \hspace{4pt} M\ar[r] & \Phi(M)^\vee.
}$$
We will first show that for any compact $\mf{g}$-representation $E$, the unit map $E\rt \nu\mu(E)$ is an equivalence. To this end, let $f\colon 0\rt \mf{g}$ be the initial map and note that the forgetful functor $f^!\colon \Rep_\mf{g}\rt \Mod_A$ detects equivalences and preserves filtered colimits. It follows that $\Rep_\mf{g}^\omega\subseteq \Rep_\mf{g}$ is the thick subcategory generated by the $\mf{g}$-representations 
$$
f_*(F) = \mc{U}(\mf{g})\otimes_A F
$$
where $F\in \Mod_A^{!, \omega}$ is a compact object in the $\infty$-category of pro-coherent sheaves on $A$. It therefore suffices to check that the unit map $f_*F\rt \nu\mu(f_*F)$ is an equivalence for any such free $\mf{g}$-representation $f_*F$. To see this, note that the unit map is the composite
$$\xymatrix{
f_*F \ar[r] & \big((f_*F)^{\vee}\big){}^\vee\ar[r]^-{\epsilon} & \Phi\big(C^*(\mf{g}, (f_*F)^\vee)\big)^\vee
}$$
where $\epsilon$ is the unit of the adjoint pair $(\Phi, C^*)$. The PBW theorem and Lemma \ref{fm:lem:verygoodlralgebroids} imply that $\mc{U}(\mf{g})\cong \bigoplus_{n\sgeq 0} A^{\oplus k_n}[-n]$ as a graded $A$-module. By Diagram \ref{diag:inductionandcoinduction}, we find that 
$$
(f_*F)^\vee = f_!(F^\vee) = \Hom_A\big(\mc{U}(\mf{g}), F^\vee\big) =  \prod_{n\sgeq 0} (F^\vee)^{\oplus k_n}[-n].
$$
Since $F^\vee\in \Mod^!_A$ is a coherent $A$-module, hence in particular eventually connective, $(f_*F)^\vee$ is eventually connective as well. The map $\epsilon$ is then an equivalence by Corollary \ref{fm:cor:equivonconnprocoherent}. Furthermore, the map $f_*F\rt (f_*F)^{\vee\vee}$ can be identified with
$$\xymatrix{
f_*F = \bigoplus_{n\sgeq 0} F^{\oplus k_n}[-n]\ar[r] & (f_!F^\vee){}^\vee = \prod_{n\sgeq 0} (F^{\vee\vee})^{\oplus k_n}[-n].
}$$
The description of the compact objects from Proposition \ref{prop:tamecompactlygenerated} shows that $F\rt F^{\vee\vee}$ is an isomorphism. Since $A$ is a bounded cdga, $F$ is trivial in sufficiently high (homological) degrees and the above map is an isomorphism. 

We conclude that $\mu$ is fully faithful on the compact objects of $\Rep_\mf{g}^!$. The essential image of the compact objects is the smallest thick subcategory of $\Mod_{C^*(\mf{g})}$ containing the images $\mu(f_*F)$. But
$$
\mu(f_*F) \simeq C^*(\mf{g}, (f_*F)^\vee) \simeq C^*(\mf{g}, f_!(F^\vee)) \simeq f_*(F^\vee)
$$
by the commuting diagrams \eqref{diag:coinduction} and \eqref{diag:inductionandcoinduction}. The compact objects of $\Mod_A^!$ are precisely the preduals of the coherent $A$-modules, so that Lemma \ref{lem:coherentoversmallext} shows that the essential image of $\Mod_A^{!, \omega}$ under $\mu$ coincides with the coherent $C^*(\mf{g})$-modules.
\end{proof}

\subsection{Pro-coherent sheaves}\label{sec:indcohonformmod}
We will now deduce Theorem \ref{fm:thm:ind-coherentvsrepresentations} from the functoriality of the equivalence provided by Proposition \ref{prop:dualityforreps}.
\begin{construction}\label{cons:indcohonformalmoduli}
It follows from Lemma \ref{lem:coherentoversmallext} that for any map $f\colon B\rt B'$ between small extensions of $A$ (by coherent $A$-modules), the functor
$$\xymatrix{
f_*\colon \Mod_{B'}\ar[r] & \Mod_B
}$$
preserves coherent modules. Consequently, there is a functor 
$$\xymatrix{
\cat{Coh}\colon \big(\CAlg^{\mm{sm, coh}}/A\big)^{\op}\ar[r] & \cat{StCat}^{\mm{ex}}_\infty
}$$
sending each small extension $B\rt A$ to the stable $\infty$-category of coherent $B$-modules, and each map $f\colon B\rt B'$ to the exact functor $f_*$. Consider the composite functor
$$\xymatrix{
\big(\CAlg_k^\mm{sm, coh}/A\big)^\op\ar[r]^-{\cat{Coh}^\op} & \cat{Cat}_\infty\ar[r]^{\mm{Ind}} & \cat{Pr}^\mm{L}
}$$
sending each small extension $B\rt A$ to the locally presentable stable $\infty$-category $\mm{Ind}(\cat{Coh}^\op_B)$ of pro-coherent sheaves on $B$. Every map $f\colon B\rt B'$ of small extensions over $A$ induces an adjoint pair
$$\xymatrix{
f_* = \mm{Ind}(f_*)\colon \mm{Ind}(\cat{Coh}^\op_{B'})\ar@<1ex>[r] & \mm{Ind}(\cat{Coh}_{B}^\op) \ar@<1ex>[l] \colon f^!.
}$$
By construction, the functor $f_*$ preserves compact objects. It follows that the right adjoint $f^!$ preserves all colimits and admits a further right adjoint $f_!$ (this is discussed in much greater generality in \cite{gai13}).
\end{construction}
\begin{definition}\label{def:procohonformmod}
For any functor $X\colon \CAlg_k^\mm{sm, coh}/A\rt \sS$, define the stable presentable $\infty$-category $\QC^!(X)$ of \emph{pro-coherent sheaves} on $X$ to be the value on $X$ of the left Kan extension
$$\xymatrix@C=3pc{
\big(\CAlg_k^\mm{sm, coh}/A\big)^\op\ar[r]^-{\mm{Ind}(\cat{Coh}^\op)}\ar[d] & \cat{Pr^\mm{L}}\\
\Fun(\CAlg_k^\mm{sm, coh}/A, \sS)\ar@{..>}[ru]_-{\QC^!}
}$$
This exists by \cite[Lemma 5.1.5.5]{lur09}, since $ \cat{Pr^\mm{L}}$ has all small colimits \cite{lur09}. Informally, a pro-coherent sheaf on $X$ is given by a collection of pro-coherent sheaves $F_y$ over $B\in \CAlg_k^\mm{sm, coh}/A$ for every $y\in X(B)$, together with a coherent family of equivalences for every $f\colon B\rt B'$
$$
F_{f(y)}\simeq f^!F_y.
$$
\end{definition}
\begin{proof}[Proof (of Theorem \ref{fm:thm:ind-coherentvsrepresentations})]
By Proposition \ref{prop:dualityforreps}, the natural transformation $\mu$ \eqref{diag:mu} induces a natural equivalence
$$\xymatrix@C=3pc{
\LieAlgd^{!, \mm{good}}_A\ar@/^1.5pc/[r]^{\Rep_\ast^{!,\omega}}_{}="s" \ar@/_1.5pc/[r]_{\cat{Coh}^{\op} \circ C^*}^{}="t" & \cat{Cat}_\infty \ar@{=>}"s";"t"^{\mu}
}$$
to the $\infty$-category of small $\infty$-categories. Here $\Rep_\ast^{!, \omega}$ sends a good pro-coherent Lie algebroid $\mf{g}$ to the $\infty$-category of compact $\mf{g}$-representations and a map $f\colon \mf{g}\rt \mf{h}$ to the induction functor $f_*$. Since each $\Rep^!_\mf{g}$ is a compactly generated stable $\infty$-category, we obtain an induced natural equivalence
$$\xymatrix@C=3pc{
\LieAlgd^{!, \mm{good}}_A\ar@/^1.5pc/[r]^{\Rep^!_\ast}_{}="s" \ar@/_1.5pc/[r]_{\mm{Ind}(\cat{Coh}^{\op}) \circ C^*}^{}="t" & \cat{Pr}^\mm{L} \ar@{=>}"s";"t"^{\mu}
}$$
between Ind-completions. Precomposing with the duality functor $\mf{D}$ \eqref{diag:deformationadjunction}, we obtain a natural diagram
$$\vcenter{\xymatrix{
\big(\CAlg_k^\mm{sm, coh}/A\big)^{\op}\ar[d]\ar@/^1.5pc/[rrd]^-{\mm{Ind}(\cat{Coh}^\op)}_{}="s" \ar[d] & & \\
\Fun(\CAlg_k^\mm{sm, coh}/A, \sS)\ar[r]_-{\mf{D}_!} & \LieAlgd^!_A\ar[r]_-{\Rep^!_\ast} & \cat{Pr^\mm{L}} \ar@{=>}"s"+<-5pt, -5pt>;[0, -2]+<50pt, 15pt>^\Psi
}}$$
commuting up a natural equivalence $\Psi$, given by \emph{inverse} of the composite
$$\vcenter{\xymatrix{
\Rep^!_{\mf{D}(B)}\ar[r]^-{\mu} &  \mm{Ind}(\cat{Coh}^\op_{B}) \ar[r] & \mm{Ind}(\cat{Coh}^\op_{C^*\mf{D}(B)})
}}$$
where the second functor arises from the unit map $B\rt C^*\mf{D}(B)$. By the universal property of $\QC^!(X)$, we obtain a natural transformation
$$\xymatrix{
\Psi_X\colon \QC^!(X)\ar[r] & \Rep_{\mf{D}_!(X)}
}$$
for any $X\colon \CAlg_k^\mm{sm}/A\rt \sS$. The proof of Theorem \ref{fm:thm:quasi-coherentvsrepresentations} can now be repeated verbatim: when $X$ is a formal moduli problem, $\mf{D}_!(X)$ is equivalent to $T_{A/X}$. Since $T_{A/X}$ is a sifted colimit of good pro-coherent Lie algebroids, $X$ is a sifted colimit of corepresentable functors. The functor $\Rep_*\colon \LieAlgd_A\rt \cat{Pr}^\mm{L}$ preserves colimits by Lemma \ref{lem:propertiesoftamereps}, so that $\Psi_X$ is a sifted colimit of equivalences and hence an equivalence itself.
\end{proof}

\bibliographystyle{abbrv}
\bibliography{../bibliography_thesis}

\end{document}